\documentclass[11pt,a4paper]{article}

\usepackage[T1]{fontenc}
\usepackage[utf8]{inputenc}
\usepackage[margin=2.7cm]{geometry}
\usepackage{amsmath,amssymb,amsthm}
\usepackage{booktabs}
\usepackage{algorithm}
\usepackage{algpseudocode}
\usepackage{graphicx}
\usepackage[colorlinks=true,citecolor=blue,linkcolor=blue,urlcolor=blue]{hyperref}
\usepackage[authoryear,longnamesfirst]{natbib}

\newtheorem{theorem}{Theorem}[section]
\newtheorem{lemma}[theorem]{Lemma}
\newtheorem{definition}[theorem]{Definition}
\newtheorem{proposition}[theorem]{Proposition}
\newtheorem{corollary}[theorem]{Corollary}
\newtheorem{example}[theorem]{Example}
\newtheorem{problem}[theorem]{Problem}

\title{Fundamentals of the cocyclic development of Hadamard matrices over loops}

\author{
Ra\'ul M. Falc\'on\thanks{Corresponding author. Department of Applied Mathematics I, Universidad de Sevilla, Spain. \texttt{rafalgan@us.es}}
\and
Manuel Gonz\'alez-Regadera\thanks{Department of Applied Mathematics I, Universidad de Sevilla, Spain. \texttt{mgonzalez34@us.es}}
\and
Felix Gudiel\thanks{Department of Applied Mathematics I, Universidad de Sevilla, Spain. \texttt{gudiel@us.es}}
}

\date{}

\begin{document}

\maketitle

\begin{abstract}
This paper establishes a formal framework for the cocyclic development of Hadamard matrices over loops based on a novel cohomology theory with associativity obstructions. In this context, we formalize the notion of a loop-cocyclic Hadamard matrix, so that both the classical cocyclic Hadamard matrices over groups and the recently introduced pseudococyclic Hadamard matrices over loops are naturally embedded within this broader algebraic architecture. To validate the computational viability of this framework, we prove the existence of four new Hadamard equivalence classes of order 24 and eight new classes of order 28 that are strictly non-cocyclic over any finite group, arising uniquely as loop-cocyclic developments over Moufang and right Bol loops.
\end{abstract}



\section{Introduction and preliminaries}
\label{Section:Intro}

A {\em Hadamard matrix} $H$ of order $n$ is an $n\times n$ array with entries in the cyclic group $\langle\,-1\,\rangle$ of order $2$ generated by $-1$ such that $HH^T=nI_n$, where $I_n$ is the identity matrix of order $n$. This last condition implies that $n$ is $1,\,2$ or a multiple of $4$. A Hadamard matrix is {\em normalized} if all the elements in its first row and first column are $1$. The {\em Hadamard Conjecture} indicates that there is a Hadamard matrix of order $4t$ for every positive integer $t$.

There are different constructions of Hadamard matrices (see \citealp{Hor07}). Here we highlight the {\em cocyclic development} of Hadamard matrices over finite groups \citep{HdL95}. In what follows, we recall some preliminaries on this development. Let $G$ be a finite group over which an abelian group $A$ is a trivial $G$-module, and let $k$ be a positive integer. (In the context of Hadamard matrices, $A$ is usually the cyclic group $\langle\,-1\,\rangle$ or the abelian group $(\mathbb{F}_2,+)$. From now on, we consider $A=\langle\,-1\,\rangle$.) A function $f:G^k\to A$ is a {\em $k$-cochain}. The set $\mathcal{C}^k(G,A)$ of $k$-cochains is an abelian group under the pointwise product. The {\em coboundary operator}
$\delta$ maps each $f\in \mathcal{C}^k\left(G,A\right)$ to a {\em $(k+1)$-coboundary} $\delta_f:=\delta(f)\in\mathcal{C}^{k+1}\left(G,A\right)$ so that
\begin{equation}\label{eq:coboundary}
\delta_f(a_1,\ldots,a_{k+1}):=f(a_1,\ldots,a_k)f(a_2,\ldots,a_{k+1})\prod_{i=1}^kf(a_1,\ldots,a_ia_{i+1},\ldots,a_{k+1})
\end{equation}
for all $a_1,\ldots,a_{k+1}\in G$. If $\delta_f={\bf 1}$, then $f$ is a {\em $k$-cocycle}. The set $\mathcal{B}^k(G,A)$ of $k$-coboundaries and the set $\mathcal{Z}^k(G,A)$ of $k$-cocycles are also abelian groups under the pointwise product.

This paper focuses on $2$-cocycles. From \eqref{eq:coboundary}, $f\in \mathcal{Z}^2(G,A)$ if and only if
\begin{equation}\label{HdLCoc}
f(a,b)f(ab,c)=f(a,bc)f(b,c)
\end{equation}
for all $a,b,c\in G$. The $2$-cocycle $f$ is uniquely determined by the {\em cocyclic matrix} $M_f:=\left(f(a,b)\right)_{a,b\in G}$. The quotient group
\[\mathcal{H}^2(G,A):=\mathcal{Z}^2(G,A)/\mathcal{B}^2(G,A)\]
is the {\em second cohomology group of $G$ over $A$}. Two cocycles $\psi_1,\psi_2\in\mathcal{Z}^2\left(G,A\right)$ are termed {\em cohomologous} if $\psi_2=\delta_f\psi_1$ for some $f\in \mathcal{C}^1\left(G,A\right)$. Every cohomology class $[\psi]\in \mathcal{H}^2(G,A)$, with representative $\psi$, endows the set $E_\psi:=A\times G$ of a group structure by means of the product
\begin{equation}\label{eq:extension}
(i,a)\cdot (j,b)=(ij\psi(a,b),ab)
\end{equation}
for all $i,j\in A$ and $a,b\in G$. It describes a short exact sequence $1\rightarrow A  \overset{\iota}{\hookrightarrow} E_\psi \xrightarrow{\pi} G\rightarrow 1$. The pair $(E_\psi,\pi)$ is termed a {\em central extension of $G$ by $A$}.

A Hadamard matrix $H$ is {\em cocyclic over $G$} if there is a pair $(f,\psi)\in \mathcal{C}^1(G,A)\times \mathcal{Z}^2(G,A)$ such that $H$ coincides with the {\em cocyclic matrix} $M_{\psi,f}:=\left(M_{f,\psi}[a,b]\right)_{a,b\in G}$ that is defined so that
\begin{equation}\label{eq:cocyclicHM}
M_{f,\psi}[a,b]:=f(ab)\psi(a,b)
\end{equation}
for all $a,b\in G$. This matrix is {\em pure cocyclic} if $f={\bf 1}$, in which case it is simply denoted by $M_\psi$. \'O Cath\'ain and R\"oder \citep{OC11} enumerated all cocyclic Hadamard matrices of order less than 40. Thus, all Hadamard matrices of order up to 20 are {\em equivalent} to a cocyclic one. That is, they are equal up to permutations or negations of rows and columns. For higher orders, 16 of the 60 Hadamard equivalence classes of order 24, and 6 of the 487 classes of order 28 are cocyclic. In addition, there are 100 cocyclic Hadamard equivalence classes of order 32, and 35 ones of order 36.

The {\em cocyclic Hadamard test} (see \citealp{HdL95}) ensures that $M_\psi$ is Hadamard if and only if the summation of all entries in each row is zero, except for the constant row. Thus, determining whether a cocyclic matrix is Hadamard is computationally much faster than the general case. The {\em cocyclic Hadamard Conjecture} indicates indeed that there is a cocyclic Hadamard matrix of order $4t$ for every positive integer $t$. Many families of Hadamard matrices are cocyclic over certain groups, such as Sylvester \citep{Sylvester1867}, Paley \citep{Paley1933}, Williamson \citep{Williamson1944}, or Ito's type Q Hadamard matrices \citep{Ito1997}. Much more recently, Baliga and Horadam \citep{Baliga1995} dealt with the existence of cocyclic Hadamard matrices over the group $\mathbb{Z}_t\times\mathbb{Z}_2^2$, with $t$ odd. For $t\leq 11$, this existence was constructively proved by Flannery \citep{Flannery1997} over the dihedral group. In the same paper, he also established some sufficient conditions for the existence of cocyclic Hadamard matrices over abelian, dicyclic, and dihedral groups of certain orders. Alternative constructions were obtained by \'Alvarez et al. \citep{Alvarez2016}. On the contrary, \'O Cath\'ain \citep{OC08} proved that some families of Hadamard matrices are not cocyclic. Thus, he proved the existence of two Goethals-Seidel Hadamard matrices \citep{GS67} of order 28, and also 15 of order 36, that are not cocyclic.

A natural generalization of this cocyclic development from groups to loops was described in \citep{Alvarez2019,Alvarez2020,Falcon2021}. Recall here that a {\em loop} is a finite set $L$ endowed with a binary operation with an identity element such that both equations $ax=b$ and $ya=b$ have unique solutions $x,y\in L$ for all $a,b\in L$. This implies the existence of both a left-division $\backslash$ and a right-division $/$. More precisely, $x=a\backslash b$ and $y=b/a$. A loop is a group whenever its binary operation is associative. The multiplication table of every loop is a {\em Latin square}. That is, each symbol appears exactly once per row, and exactly once per column. A family of loops of particular relevance in our study is that of Moufang loops. A {\em Moufang loop} \citep{Moufang1935} is a loop $L$ that satisfies $a(b(ac))=((ab)a)c$ for all $a,b,c\in L$. Moufang loops preserve some relevant results in group theory such as the Lagrange theorem. Moreover, they are diassociative. That is, any two of their elements generate a group. Hence, Moufang loops are very good candidates for a cocyclic development of Hadamard matrices.

In the literature, cohomology of loops has been dealt with by means of either iterations of the {\em associator} $\mathcal{A}(a,b,c):=((ab)c)/(a(bc))$ \citep{Eilenberg1947} or module extensions \citep{Johnson1990}. Despite this, the cocyclic development of Hadamard matrices over loops was introduced by \'Alvarez et al. \citep{Alvarez2019} without delving into any of these two loop cohomologies. Cocycles, cochains and coboundaries of a loop $L$ were introduced exactly as the homologous concepts for groups, once the underlying group is replaced by $L$. Under this assumption, the respective sets $\mathcal{Z}^k(L,A)$, $\mathcal{C}^k(L,A)$ and $\mathcal{B}^k(L,A)$ of $k$-cocycles, $k$-cochains and $k$-coboundaries of $L$ are again abelian groups under the pointwise product. The cocyclic Hadamard test for groups is also valid for non-associative loops (see \citealp[Theorem~29]{Alvarez2019}).

There are mere a pair of results concerning the existence of cocyclic Hadamard matrices over non-associative loops. Thus, the first order for which one such a matrix exists is eight (see \citealp[Theorem~30]{Alvarez2019}). In addition, the direct product of two loops (at least one of them being non-associative) over which a cocyclic Hadamard matrix exists is also a non-associative loop with the same property (see \citealp[Theorem~31]{Alvarez2019}).

The main contribution of this paper on this topic is a positive answer to the following open problem.

\begin{problem}[Based on Problem 43 in \citealp{Alvarez2019}] \label{problem3a} Does there exist a Hadamard equivalence class that is not cocyclic over a group, but it is cocyclic over a non-associative loop?
\end{problem}

Our affirmative answer to Problem \ref{problem3a} illustrates that the cocyclic Hadamard Conjecture is highly constrained by the associative law, so the cocyclic development over loops may be the key to delve into this conjecture. Our approach is closer to the original cohomology of loops described by Eilenberg and MacLane, because we introduce a non-associative perturbation in the cocyclic development of Hadamard matrices over groups. To understand much better the necessity of this perturbation, a relevant aspect here concerns the inconveniences of preserving the definition \eqref{eq:coboundary} on coboundaries from groups to non-associative loops. Unlike the cocyclic development over groups, preserving that definition does not ensure that every coboundary is a cocycle. More precisely, if $(f,\psi)\in \mathcal{C}^1(L,A)\times\mathcal{Z}^2(L,A)$, then
\begin{equation}\label{eq_pseudo}
\delta_f\psi\in \mathcal{Z}^2(L,A)\Leftrightarrow f((ab)c)=f(a(bc)) \text{ for all } a,b,c\in L.
\end{equation}
This condition is too restrictive and would imply $f=\mathbf{1}$ in general. As an alternative, the notion of {\em pseudocoboundary} was introduced in \citep{Alvarez2020} (see also \citealp{Falcon2021}) as an elementary coboundary $\delta_{\partial_h}\in \mathcal{B}^2(L,A)$, with $h\in L$, where $\partial_h\in \mathcal{C}^1(L,A)$ is defined so that $\partial_h(s)=1$ if $h\neq s$, $\partial_h(h)=-1$, and
\begin{equation}\label{eq:pseudocoboundary}
\partial_h((ab)c)\neq \partial_h(a(bc)) \text{ for some } a,b,c\in L.
\end{equation}
Then, the notion of a {\em pseudococycle} was also introduced in \citep{Alvarez2020} as a product
\begin{equation}\label{eq:pseudococycle}\left(\prod_{s\in S}\delta_{\partial_s}\right)\psi\in \mathcal{Z}^2(L,A),
\end{equation}
with $S\subseteq L$ and $\psi\in \mathcal{Z}^2(L,A)$. A Hadamard matrix is then termed {\em pseudococyclic} if it is equivalent to the matrix defined by any such pseudococycle. Unlike cocyclic Hadamard matrices over groups and loops, the Hadamard cocyclic test is not valid for pseudococylic Hadamard matrices (see \citealp[Example~4]{Alvarez2020}). Every Hadamard matrix of Goethals-Seidel type (some of which are non-cocyclic developed over any group \citealp{OC08}) is pseudococyclic (see \citealp[Theorem~1]{Alvarez2020}). This fact justifies the interest of dealing with the cocyclic development of Hadamard matrices over non-associative loops.

The paper is organized as follows. In Section \ref{sec:preliminaries}, we make use of a corrector factor that arises from \eqref{eq_pseudo} to introduce loop-cocycles as a natural generalization of cocycles of loops. We prove that the set of coboundaries is a normal subgroup of the group of loop-cocycles. Based on this normal subgroup, we define a quotient group, and hence an equivalence relations among loop-cocyles. This enables us to introduce in Section \ref{sec:hadamard} the notion of loop-cocyclic Hadamard matrices over loops as a natural generalizations of cocyclic Hadamard matrices over groups. In particular, every pseudococyclic Hadamard matrix is a loop-cocyclic Hadamard matrix.

\section{Loop-cocycles}\label{sec:preliminaries}

In this section, we describe the cohomology with associativity obstruction that serves as a formal framework for the cocyclic development of Hadamard matrices over loops. Throughout the section, we consider a loop $L$ of order $n$ with identity element $e$. In addition, we represent the cocycle values in the additive field $\mathbb{Z}_2:=(\mathbb{F}_2,+)$ rather than in the multiplicative set $\left\langle\,-1\,\right\rangle$. With this choice, we shift from the combinatorial complexity of sign multiplication to the structural clarity of linear subspaces. This additive representation is essential in our illustrative examples for identifying bases of the corresponding cocycle vector space and managing the associativity obstructions through vector addition.

The following definition introduces the notion of a loop-cocycle as a natural generalization of the cocycles described in \citep{Alvarez2019}.

\begin{definition}\label{def_f_cocycle} A {\em $k$-loop-cocycle} of $L$ over $\mathbb{Z}_2$ is a $k$-cochain $\psi\in\mathcal{C}^k(L,\mathbb{Z}_2)$ such that $\delta_\psi\in\mathrm{Im}(\delta^2)$. We denote by $\mathcal{Z}_{\mathcal{L}}^k(L,\mathbb{Z}_2)$ the abelian group under the pointwise product of $k$-loop-cocycles of $L$ over $\mathbb{Z}_2$.
\end{definition}

The classical notion of a $k$-cocycle over a group arises naturally when the loop $L$ is associative because in that case $\delta_\psi=\delta^2(0)=0$. However, in the non-associative case, the coboundary operator is now measuring not just exactness but also the lack of associativity as an obstruction that prevents the cochain from closing into a cocycle. In this context, $\delta^2$ acts as a curvature operator just as in differential geometry, where the second derivative does not vanish whenever the space is not flat.

In general, $\mathcal{B}^k(L,\mathbb{Z}_2)\not \subseteq \mathcal{Z}^k(L,\mathbb{Z}_2)$ for non-associative loops (see \citealp{Alvarez2020}). The next lemma shows the right framework for loop-cocycles. It follow readily from Definition \ref{def_f_cocycle}.

\begin{lemma}\label{lemma_subgroup} It is verified that
\[\mathcal{B}^k(L,\mathbb{Z}_2)\trianglelefteq \mathcal{Z}_{\mathcal{L}}^k(L,\mathbb{Z}_2) \hspace{1cm}\text{and}\hspace{1cm} \mathcal{Z}^k(L,\mathbb{Z}_2)\trianglelefteq \mathcal{Z}_{\mathcal{L}}^k(L,\mathbb{Z}_2).\]
\end{lemma}

As an immediate consequence, the pseudococyclic development described in the introduction is naturally embedded in the loop-cocyclic development.

\begin{proposition}\label{proposition_pseudocoboundary} Every pseudocoboundary and every pseudococycle of $L$ is a loop-cocycle.
\end{proposition}

The following example illustrates the relevance of this pseudococyclic development. It gives an affirmative answer to Problem \ref{problem3a}.

\begin{example}\label{example_GS28} A {\em Goethals-Seidel array} \citep{GS67} of order $4t$ is a $4t \times 4t$-block matrix
\[\left( \begin{array}{rrrr} A& BR & CR& DR\\ BR &-A& RD& -RC\\ CR&-RD& -A& RB\\ DR& RC& -RB& -A \end{array} \right)\]
where $A$, $B$, $C$ and $D$ are $t\times t$-circulant matrices, and $R$ is the back diagonal matrix of order $t$. This array is Hadamard if
\[AA^T+BB^T+CC^T+DD^T=4t\cdot I_{t}\]
where $X^T$ denotes the transpose of $X\in\{A,B,C,D\}$ and $I_t$ denotes the identity matrix of order $t$.

Based on the construction described in \citep{Alvarez2020}, we have computed in the {\sc GAP} system {\em (Groups, Algorithms, Programming)} \citep{GAP} all the Hadamard equivalence classes of order 28 that are pseudococyclic over the Moufang loop $\mathrm{GS}_{28}$ therein introduced, none of which are cocyclic over groups. (Note that there is a unique Moufang loop of order $28$ up to isomorphism \citealp{Chein1974}.) We have obtained the four Hadamard equivalence classes described in Table \ref{tab:28}. Each class is labeled according to the Sloane's library of Hadamard matrices \citep{Sloane}. For each class, we indicate the size of its automorphism group, and the signs of the first rows of the circulant matrices $A$, $B$, $C$ and $D$ that describe the Goethals-Seidel Hadamard array under consideration. Note here that O'Cath\'ain already realized in his Ph.D. Thesis \citep{OC08} the non-cocyclic development of both Goethals-Seidel arrays \texttt{had.28.2} and \texttt{had.28.375}.

These four Hadamard equivalence classes over the Moufang loop $\mathrm{GS}_{28}$ are in addition to the known six cocyclic Hadamard classes over groups, making a new total (for the moment) of 10 cocyclic classes out of the 487 existing ones.
\end{example}

\begin{table}[ht]
\centering
\caption{Goethals-Seidel Hadamard equivalence classes that are pseudococyclic over the Moufang loop $\mathrm{GS}_{28}$ described in \citep{Alvarez2020}.}
\label{tab:28}
\begin{tabular}{@{}lcllll}
\toprule
\textbf{ID \citep{Sloane}} &  \textbf{$|\mathrm{Aut}(L)|$} & $A$ & $B$ & $C$ & $D$ \\ \midrule
\texttt{had.28.2} & 24 & [-\,-\,-\,-\,-\,+\,+] & [-\,-\,-\,-\,+\,-\,+] & [-\,-\,-\,+\,-\,-\,+] & [-\,-\,-\,+\,-\,+\,+]\\
\texttt{had.28.18} & 18 & [-\,-\,-\,-\,-\,-\,+] & [-\,-\,-\,+\,-\,+\,+] & [-\,-\,-\,+\,-\,+\,+]& [-\,-\,-\,+\,-\,+\,+]\\
\texttt{had.28.375} & 48 & [-\,-\,-\,-\,-\,-\,+] & [-\,-\,-\,-\,+\,+\,+] & [-\,-\,+\,-\,-\,+\,+]& [-\,-\,+\,-\,+\,-\,+]\\
\texttt{had.28.410} & 16 & [-\,-\,-\,-\,-\,+\,+] & [-\,-\,-\,-\,-\,+\,+] & [-\,-\,-\,+\,-\,-\,+] & [-\,-\,+\,-\,+\,-\,+]\\
\bottomrule
\end{tabular}
\end{table}

In what follows, we describe a cohomology with associativity obstructions that is based on the classical loop cohomology introduced by Eilenberg and MacLane in \citep{Eilenberg1947}. We will show how this cohomology is useful to deal with the enumeration of loop-cocycles. First, we define the {\em associative obstruction set}
\[\mathcal{O}(L):=\bigcup_{a,b,c\in L}\left\{(ab)c,\, a(bc)\colon\, \mathcal{A}(a,b,c)\neq e\right\}\]
where $\mathcal{A}(a,b,c)$ is the associator. Then, for each positive integer $k\geq 2$, we define the abelian subgroup $\mathcal{C}_0^{k-1}(L,\mathbb{Z}_2)\trianglelefteq \mathcal{C}^{k-1}(L,\mathbb{Z}_2)$ of cochains $f\in \mathcal{C}^{k-1}(L,\mathbb{Z}_2)$ such that $f(a_1,\ldots,a_{i-1},a,a_{i+1},\ldots,a_{k-1})=0$ for all $a_1,a_{i-1},a_{i+1},\ldots,a_{k-1}\in L$, with $i\in\{1,\ldots,k-1\}$, and $a\not\in\mathcal{O}(L)$. In particular,
\[\mathcal{C}_0^1(L,\mathbb{Z}_2):=\left\{f\in \mathcal{C}^1(L,\mathbb{Z}_2)\colon\, f(a)=0 \text{ for all } a\not\in\mathcal{O}(L)\right\}\trianglelefteq \mathcal{C}^1(L,\mathbb{Z}_2).\]
We also define the abelian subgroup
\[\mathcal{B}_0^k(L,\mathbb{Z}_2):=\left\{\delta_f\in \mathcal{B}^k(L,\mathbb{Z}_2)\colon\, f\in \mathcal{C}^{k-1}_0(L,\mathbb{Z}_2)\right\}\trianglelefteq \mathcal{B}^k(L,\mathbb{Z}_2)\]
and the quotient group
\[\mathcal{B}_\mathcal{L}^k(L,\mathbb{Z}_2):=\mathcal{B}_0^k(L,\mathbb{Z}_2)/\left(\mathcal{B}_0^k(L,\mathbb{Z}_2)\cap \mathcal{Z}^k(L,\mathbb{Z}_2)\right).\]
From now on, we denote by $[\delta_f]$ the equivalence class in $\mathcal{B}_\mathcal{L}^k(L,\mathbb{Z}_2)$ of a coboundary $\delta_f\in \mathcal{B}_0^k(L,\mathbb{Z}_2)$. In particular, if $L$ is a group, then
\[\mathcal{O}(L)=\emptyset,\hspace{1cm} \mathcal{C}_0^{k-1}(L,\mathbb{Z}_2)=\mathcal{B}_0^k(L,\mathbb{Z}_2)=\{0\}\hspace{1cm}\text{and}\hspace{1cm} \mathcal{B}_\mathcal{L}^k(L,\mathbb{Z}_2)=\{[0]\}.\]
The following lemma shows that $\mathcal{B}_\mathcal{L}^k(L,\mathbb{Z}_2)\cong \mathrm{Im}(\delta^2)$, the image of $\delta^2$.

\begin{lemma}\label{lemma_quotient}
The map
\[\begin{array}{cccc}
\Phi: & \mathcal{B}_{\mathcal{L}}^k(L,\mathbb{Z}_2)& \to & \mathrm{Im}(\delta^2)\\
& [\delta_f] &\mapsto & \delta^2(f)
\end{array}\]
is an isomorphism of abelian groups.
\end{lemma}

\begin{proof} First, we claim that $\Phi$ is well-defined. To prove it, let $f,g\in \mathcal{C}_0^{k-1}(L,\mathbb{Z}_2)$ be such that $[\delta_f]=[\delta_g]$. Then, $\delta_{f-g}=\delta_f-\delta_g\in \mathcal{Z}^k(L,\mathbb{Z}_2)$. Thus, $\delta^2(f-g)=0$ and hence $\delta^2(f)=\delta^2(g)$.

Concerning injectivity, if
$\delta^2(f)=\delta^2(g)$ for a pair of coboundaries $\delta_f,\delta_g\in \mathcal{B}_0^k(L,\mathbb{Z}_2)$, then $\delta^2(f-g)=0$ and hence $\delta_f-\delta_g=\delta_{f-g}\in \mathcal{Z}^k(L,\mathbb{Z}_2)$. That is, $[\delta_f]=[\delta_g]$.

Now, we prove that $\Phi$ is onto. Let $\delta^2_f\in \mathrm{Im}(\delta^2)$ be the image by $\delta^2$ of a cochain $f\in\mathcal{C}^{k-1}(L,\mathbb{Z}_2)$. Let $f_0\in\mathcal{C}_0^{k-1}(L,\mathbb{Z}_2)$ be defined so that $f_0(a)=f(a)$ for all $a\in\mathcal{O}(L)$, and $f_0(a)=0$, otherwise. Then, $\delta_{f_0}\in \mathcal{B}_0^k(L,\mathbb{Z}_2)$, and $\delta^2_f=\delta^2_{f_0}$. That is, $\Phi([\delta_{f_0}])=\delta^2_f$.

Then, $\Phi$ is a group isomorphism because
\[\Phi(\delta_f+\delta_g)=\Phi(\delta_{f+g})=\delta^2(f+g)=\delta^2(f)+\delta^2(g)=\Phi(\delta_f)+\Phi(\delta_g)\]
for all $f,g\in \mathcal{C}_0^{k-1}(L,\mathbb{Z}_2)$.
\end{proof}

The next result shows the relationship between the abelian groups $\mathcal{Z}_{\mathcal{L}}^k(L,\mathbb{Z}_2)$ and $\mathcal{B}_{\mathcal{L}}^k(L,\mathbb{Z}_2)$. It follows readily from Definition \ref{def_f_cocycle} and Lemma \ref{lemma_quotient}.

\begin{proposition}\label{proposition_f_cocycle} The map
\[\begin{array}{cccc}
\Phi^{-1}\delta: & \mathcal{Z}_{\mathcal{L}}^k(L,\mathbb{Z}_2) & \to & \mathcal{B}_{\mathcal{L}}^k(L,\mathbb{Z}_2)\\
& \psi &\mapsto & \Phi^{-1}\delta(\psi):=\Phi^{-1}(\delta_\psi)
\end{array}\]
is a surjective group homomorfism with $\mathrm{ker}(\Phi^{-1}\delta)=\mathcal{Z}^k(L,\mathbb{Z}_2)$.
\end{proposition}

Lemma \ref{lemma_subgroup} and Proposition \ref{proposition_f_cocycle} allow us to show how the category of finite loops admits a cohomology with associativity obstructions.

\begin{theorem}\label{theorem_family_cochains} Loop-cocycles describe the short exact sequence
\begin{equation}\label{eq:sequence}
0\rightarrow \mathcal{Z}^k(L,\mathbb{Z}_2)\overset{\iota}{\hookrightarrow} \mathcal{Z}^k_\mathcal{L}(L,\mathbb{Z}_2) \xrightarrow{\Phi^{-1}\delta} \mathcal{B}_\mathcal{L}^k(L,\mathbb{Z}_2)\rightarrow  0
\end{equation}
where $\iota$ is the natural inclusion.
\end{theorem}

\begin{corollary}\label{corollary_family_cochains_0} It is verified that
\[\mathcal{Z}^k(L,\mathbb{Z}_2)\cong\mathcal{Z}_{\mathcal{L}}^k(L,\mathbb{Z}_2)/\mathcal{B}_{\mathcal{L}}^k(L,\mathbb{Z}_2).\]
\end{corollary}

The abelian groups $\mathcal{Z}^k(L,\mathbb{Z}_2)$ and $\mathcal{Z}_{\mathcal{L}}^k(L,\mathbb{Z}_2)$ can be considered as vector spaces over the Galois field $\mathbb{F}_2$, once each loop-cocycle $\psi\in\mathcal{Z}_{\mathcal{L}}^k(L,\mathbb{Z}_2)$ is represented by a solution of the homogeneous linear system of equations described by the condition $\delta_\psi\in \mathrm{Im}(\delta^2)$. The following lemma illustrates this fact for $k=2$.

\begin{lemma}\label{lemma_zeros} The loop-cocycles in $\mathcal{Z}_{\mathcal{L}}^2(L,\mathbb{Z}_2)$ are identified with the solutions in $\mathbb{Z}_2^{n^2+n}$ of the homogeneous linear system of equations
\begin{equation}\label{eq:system}
\begin{cases}
\begin{array}{rl}x_{a,b}+x_{ab,c}+x_{a,bc}+x_{b,c}+x_{(ab)c}+x_{a(bc)}=0, & \text{ for all }
a,b,c\in L,\\
x_a=0, & \text{ for all } a\in L\setminus\mathcal{O}(L).
\end{array}
\end{cases}
\end{equation}
over $\mathbb{Z}_2$ in the set of variables $\left\{x_{a,b}\colon\, a,b\in L\right\}\cup \left\{x_a\colon\,  a\in L\right\}$. In particular, the subset of solutions that also satisfy $x_a=0$ for all $a\in \mathcal{O}(L)$ identifies the set $\mathcal{Z}^2(L,\mathbb{Z}_2)$.
\end{lemma}

\begin{proof} Every solution of \eqref{eq:system} is a vector $s:=\left(s_{0,0},\ldots,s_{n-1,n-1},s_0,\ldots,s_{n-1}\right)\in \mathbb{Z}_2^{n^2+n}$ where each $s_{a,b}$ is related to the variable $x_{a,b}$, and each $s_a$ is the component related to the variable $x_a$. More precisely, the vector $s$ is uniquely identified with a loop-cocycle $\psi_s\in\mathcal{Z}_\mathcal{L}^2(L,\mathbb{Z}_2)$, where $\psi_s(a,b):=s_{a,b}$ for all $a,b\in L$. Moreover, there is a cochain $f_s\in\mathcal{C}^1(L,\mathbb{Z}_2)$ such that $f_s(a):=s_a$ for all $a\in L$, and $[\delta_{f_s}]=\Phi^{-1}(\delta_{\psi_s})$. The last statement follows readily.
\end{proof}

As a consequence, the following result holds.

\begin{proposition} $\mathcal{B}^k_\mathcal{L}(L,\mathbb{Z}_2)$ is a vector space over $\mathbb{F}_2$.
\end{proposition}

\begin{proof} It is readily verified that $\mathcal{C}_0^{k-1}(L,\mathbb{Z}_2)$ is an $|\mathcal{O}(L)|$-dimensional vector space over $\mathbb{F}_2$. Then, $\mathcal{B}^k_0(L,\mathbb{Z}_2)$ is also a vector space over $\mathbb{F}_2$ because the coboundary operator $\delta$ is linear. Since $\mathcal{Z}^k(L,\mathbb{Z}_2)$ is also a vector space over $\mathbb{F}_2$, the result holds.
\end{proof}

Therefore, the short exact sequence \eqref{eq:sequence} splits and the following result holds.

\begin{theorem}\label{theorem_family_cochains_a} It is verified that
\begin{equation}\label{eq:decomposition}
\mathcal{Z}_\mathcal{L}^k(L,\mathbb{Z}_2)\cong \mathcal{Z}^k(L,\mathbb{Z}_2)\oplus \mathcal{B}_\mathcal{L}^k(L,\mathbb{Z}_2).
\end{equation}
As a consequence,
\[\dim_{\mathbb{F}_2}\left(\mathcal{Z}_\mathcal{L}^k(L,\mathbb{Z}_2)\right)=\dim_{\mathbb{F}_2}\left(\mathcal{Z}^k(L,\mathbb{Z}_2)\right)+\dim_{\mathbb{F}_2}\left(\mathcal{B}_\mathcal{L}^k(L,\mathbb{Z}_2)\right).\]
\end{theorem}

Although the exactness of the sequence \eqref{eq:sequence} is a universal property of the category of finite loops, the potential use of a specific loop for combinatorial designs (such as Hadamard matrices) lies exclusively in the kernel topology $\mathcal{Z}_\mathcal{L}^k(L,\mathbb{Z}_2)$. In this regard, the following result holds readily from the canonical section induced by $\Phi$ in \eqref{eq:sequence}.

\begin{corollary}\label{corollary_family_cochains} Let $\psi\in \mathcal{Z}_\mathcal{L}^k(L,\mathbb{Z}_2)$ be such that $\delta_\psi=\delta^2(f)$ for some $f\in \mathcal{C}^{k-1}(L,\mathbb{Z}_2)$. Then, $\psi$ admits the decomposition $\psi=(\psi+\delta_f)+\delta_f$, where $\psi+\delta_f\in \mathcal{Z}^k(L,\mathbb{Z}_2)$ and $\delta_f\in \mathcal{B}_0^k(L,\mathbb{Z}_2)$.
\end{corollary}

The elegance of this result lies in the role of cochains. In the group-theoretic case, the image of $s$ is absorbed into the kernel because $\delta f$ is necessarily associative. That is, the cohomology is frozen within the associative locus. However, the transition to loops acts as a decompactification of the search space. The section $s$ provides the geometric degrees of freedom necessary to escape the rigidity of the group axioms. This proves that the non-associativity of the loop-cocycle is entirely concentrated in the coboundary term. We effectively parameterize the entire loop-cocycle vector space as a linear translation of the kernel $\mathcal{Z}^k(L,\mathbb{Z}_2)$ by $\mathcal{B}^k_\mathcal{L}(L,\mathbb{Z}_2)$. This procedure unlocks up to $|\mathcal{O}(L)|^{k-1}$ additional transverse degrees of freedom that were previously hidden by the group axioms. This allows us to navigate the vast combinatorial landscape of the loop-cocycle space by means of a structured exploration of fibers.

In what follows, we show how the automorphism group $\mathrm{Aut}(L)$ of the loop $L$ makes easier the enumeration of loop-cocycles. Recall here that an {\em automorphism} of $L$ is a bijection $\phi:L\to L$ such that
\begin{equation}\label{eq_automorphism}
\phi(a)\phi(b)=\phi(ab)
\end{equation}
for all $a,b\in L$. The group $\mathrm{Aut}(L)$ acts on $\mathcal{C}^k(L,\mathbb{Z}_2)$ by $\phi\cdot f := f^\phi$ for all $\phi\in\mathrm{Aut}(L)$ and $f\in \mathcal{C}^k(L,\mathbb{Z}_2)$, where $f^\phi\in \mathcal{C}^k(L,\mathbb{Z}_2)$ is described so that
\begin{equation}\label{eq_action}
f^\phi(a_1,\ldots,a_k) := f(\phi(a_1),\ldots,\phi(a_k))
\end{equation}
for all $a_1,\ldots,a_k\in L$.

\begin{lemma}\label{lemma_aut} Every automorphism in $\mathrm{Aut}(L)$ preserves the set $\mathcal{O}(L)$ and the groups $\mathcal{C}_0^{k-1}(L)$, $\mathcal{B}_0^k(L)$ and $\mathcal{B}_L^k(L)$.
\end{lemma}

\begin{proof} Let $\phi\in \mathrm{Aut}(L)$. From \eqref{eq_automorphism}, we have
\[(ab)c\neq a(bc)\Leftrightarrow (\phi(a)\phi(b))\phi(c)\neq \phi(a)(\phi(b)\phi(c))\]
for all $a,b,c\in L$. Thus, $\phi(\mathcal{O}(L))=\mathcal{O}(L)$ and therefore, the action \eqref{eq_action} preserves the group $\mathcal{C}_0^{k-1}(L)$. Moreover, from \eqref{eq_automorphism} and the definition of the coboundary operator, we have $(\delta_f)^\phi=\delta_{f^\phi} $ and $(\delta^2(f))^\phi=\delta^2(f^\phi)$ for all $f\in \mathcal{C}^k(L)$. Thus, $\phi$ preserves the groups $\mathcal{B}_0^k(L)$ and $\mathcal{B}_L^k(L)$.
\end{proof}

\begin{proposition}\label{proposition_aut} Every automorphism in $\mathrm{Aut}(L)$ preserves the group $\mathcal{Z}_\mathcal{L}^k(L,\mathbb{Z}_2)$. Moreover, it preserves the direct sum decomposition \eqref{eq:decomposition}.
\end{proposition}

\begin{proof} Let $\phi\in\mathrm{Aut}(L)$. In addition, let $\psi\in \mathcal{Z}_\mathcal{L}^k(L,\mathbb{Z}_2)$ and $f\in \mathcal{C}^{k-1}(L,\mathbb{Z}_2)$ be such that $\delta_\psi=\delta^2(f)$. Since $\delta_{\psi^\phi}=\delta^2(f^\phi)$, we have from Lemma \ref{lemma_aut} that the action \eqref{eq_action} preserves the group $\mathcal{Z}_\mathcal{L}^k(L,\mathbb{Z}_2)$. Concerning the direct sum decomposition, we claim that the canonical section induced by $\Phi$ in Corollary \ref{corollary_family_cochains} is equivariant, as we now show. This holds because $\Phi$ is equivariant. More precisely,
\[\Phi([\delta_f]^\phi) = \Phi([\delta_{f^\phi}]) = \delta^2(f^\phi) = (\delta^2(f))^\phi = (\Phi([\delta_f]))^\phi\]
and
\[\Phi^{-1}\delta(\psi^\phi)=\Phi^{-1}(\delta_{\psi^\phi}) = \Phi^{-1}((\delta_\psi)^\phi) = (\Phi^{-1}(\delta_\psi))^\phi=(\Phi^{-1}\delta(\psi))^\phi.\]
Since the natural inclusion $\iota$ is also equivariant, we have
\[\psi^\phi =\left((\psi+\delta_f)+\delta_f\right)^\phi=\left(\psi^\phi +(\delta_f)^\phi\right)+(\delta_f)^\phi=\left(\psi^\phi +\delta_{f^\phi}\right)+\delta_{f^\phi}.\]
Thus, the $\mathcal{Z}^k(L,\mathbb{Z}_2)$-component of $\psi^\phi$ is $\psi^\phi + \delta_{f^\phi}=(\psi+\delta_f)^\phi$, and the $\mathcal{B}_\mathcal{L}^k(L,\mathbb{Z}_2)$-component is $\Phi^{-1}\delta(\psi^\phi)=(\Phi^{-1}\delta(\psi))^\phi$. Hence $\phi$ maps each summand to the corresponding summand of $\psi^\phi$. It proves the equivariance of the decomposition.
\end{proof}

\begin{theorem}\label{theorem_aut} The orbits of $\mathcal{Z}_\mathcal{L}^k(L,\mathbb{Z}_2)/\mathrm{Aut}(L)$ are in bijection with pairs $([\psi_0], [\delta_f])$ where $[\psi_0]\in \mathcal{Z}^k(L,\mathbb{Z}_2)/\mathrm{Aut}(L)$ and $[\delta_f]\in \mathcal{B}_\mathcal{L}^k(L,\mathbb{Z}_2)/\mathrm{Stab}_{\mathrm{Aut}(L)}(\psi_0)$ for any representative $\psi_0\in [\psi_0]$, where $\mathrm{Stab}_{\mathrm{Aut}(L)}(\psi_0)$ is the stabilizer group of $\psi_0$ under the action of $\mathrm{Aut}(L)$. As a consequence,
\[\left|\mathcal{Z}_\mathcal{L}^k(L,\mathbb{Z}_2)/\mathrm{Aut}(L)\right|=\sum_{[\psi_0]\in \mathcal{Z}^k(L,\mathbb{Z}_2)/\mathrm{Aut}(L)} \left|\mathcal{B}_\mathcal{L}^k(L,\mathbb{Z}_2)/\mathrm{Stab}_{\mathrm{Aut}(L)}(\psi_0)\right|.\]
\end{theorem}

\begin{proof} Two loop-cocycles $\psi_1,\psi_2\in \mathcal{Z}_\mathcal{L}^k(L,\mathbb{Z}_2)$ are in the same orbit of $\mathcal{Z}_\mathcal{L}^k(L,\mathbb{Z}_2)/\mathrm{Aut}(L)$ if and only if there exists an automorphism $\phi\in\mathrm{Aut}(L)$ such that $\psi_1^\phi = \psi_2$. By the equivariance of the direct sum decomposition proved in Proposition~\ref{proposition_aut}, $\psi_1^\phi = \psi_2$ if and only if the $\mathcal{Z}^k(L,\mathbb{Z}_2)$-component of $\psi_1$ maps to that of $\psi_2$
under $\phi$, and simultaneously the $\mathcal{B}_\mathcal{L}^k(L,\mathbb{Z}_2)$-component of $\psi_1$ maps to that of $\psi_2$ under $\phi$. The first condition is equivalent to saying that the $\mathcal{Z}^k(L,\mathbb{Z}_2)$-components of $\psi_1$ and $\psi_2$ lie in a same orbit $[\psi_0]$ of $\mathcal{Z}^k(L,\mathbb{Z}_2)/\mathrm{Aut}(L)$. Then, the second condition is equivalent to saying that their $\mathcal{B}_\mathcal{L}^k(L,\mathbb{Z}_2)$-components lie in the same orbit of $\mathcal{B}_\mathcal{L}^k(L,\mathbb{Z}_2)/\mathrm{Stab}_{\mathrm{Aut}(L)}(\psi_0)$. The consequence follows straightforwardly.
\end{proof}

\section{Loop-cocyclic Hadamard matrices}\label{sec:hadamard}

As an illustrative example, this section focuses on the case $k=2$ for dealing with the computation of Hadamard matrices that can be developed over non-associative loops. The following definition introduces the notion of a loop-cocyclic Hadamard matrix.

\begin{definition}\label{def_L_cocyclic_Hadamard}
A Hadamard matrix $H$ is termed {\em loop-cocyclic over $L$} if there is a loop-cocycle $\psi\in \mathcal{Z}_\mathcal{L}^2(L,\mathbb{Z}_2)$ such that $H$ is Hadamard equivalent to the {\em loop-cocyclic matrix} $M_\psi:=\left(M_\psi[a,b]\right)_{a,b\in L}$, where $M_\psi[a,b]:=-1^{\psi(a,b)}$ for all $a,b\in L$.
\end{definition}

The classical notion of a cocyclic Hadamard matrix arises naturally when the loop $L$ is associative. Moreover, Proposition \ref{proposition_pseudocoboundary} implies that the pseudococyclic matrices described in \citep{Alvarez2020} are indeed loop-cocyclic matrices. Thus, both developments are included in our proposal. Furthermore, the following result shows that the necessary condition in the cocyclic Hadamard test is also valid for loop-cocyclic Hadamard matrices. It generalizes Theorem 29 in \citep{Alvarez2019}.

\begin{theorem}[Cocyclic Hadamard test over a loop]\label{theo_CocHadTest} A loop-cocycle $\psi\in \mathcal{Z}_{\mathcal{L}}^2(L,\mathbb{Z}_2)$ is a Hadamard matrix only if $\sum_{b\in  L} M_\psi[a,b]=0$ whenever $a\in L\setminus\{e\}$. If $\psi\in \mathcal{Z}^2(L,\mathbb{Z}_2)$, then the converse is also valid.
\end{theorem}

\begin{proof} The necessary condition follows readily from the normality of $M_{\psi}$ and the orthogonality of the rows in all Hadamard matrices. In order to prove the second statement, let $\psi\in \mathcal{Z}^2(L,A)$ and $a,b\in L$. In addition, let $/$ be the right division on the loop $L$. Then, we have from \eqref{HdLCoc}
\begin{align*}
\sum_{c\in L}M_\psi[a,c]M_\psi[b,c] & = \sum_{c\in L}M_\psi[a,c]M_\psi[a/b,b]M_\psi[(a/b)b,c]M_\psi[a/b,bc)]=\\
& = M_\psi[a/b,b]\sum_{c\in L}M_\psi[a/b,bc]=0.
\end{align*}
\end{proof}

The following example illustrates that the converse of Theorem \ref{theo_CocHadTest} does not hold in general for loop-cocycles. This fact was already described for pseudococycles in Example 4 in \citep{Alvarez2020}.

\begin{example}\label{example_subgroup} We consider the non-associative Moufang loop $L$ of order $24$ that is described by the following Latin square.
{\tiny \[\begin{array}{|c|c|c|c|c|c|c|c|c|c|c|c|c|c|c|c|c|c|c|c|c|c|c|c|}  \hline
1 & 2 & 3 & 4 & 5 & 6 & 7 & 8 & 9 & 10 & 11 & 12 & 13 & 14 & 15 & 16 & 17 & 18 & 19 & 20 & 21 & 22 & 23 & 24 \\ \hline
2 & 3 & 5 & 6 & 1 & 7 & 9 & 10 & 4 & 11 & 12 & 8 & 14 & 15 & 17 & 22 & 13 & 23 & 24 & 18 & 20 & 19 & 21 & 16 \\ \hline
3 & 5 & 1 & 7 & 2 & 9 & 4 & 11 & 6 & 12 & 8 & 10 & 15 & 17 & 13 & 19 & 14 & 21 & 16 & 23 & 18 & 24 & 20 & 22 \\ \hline
4 & 10 & 7 & 8 & 12 & 2 & 11 & 1 & 5 & 6 & 3 & 9 & 16 & 18 & 19 & 20 & 21 & 22 & 23 & 13 & 24 & 14 & 15 & 17 \\ \hline
5 & 1 & 2 & 9 & 3 & 4 & 6 & 12 & 7 & 8 & 10 & 11 & 17 & 13 & 14 & 24 & 15 & 20 & 22 & 21 & 23 & 16 & 18 & 19 \\ \hline
6 & 11 & 9 & 10 & 8 & 3 & 12 & 2 & 1 & 7 & 5 & 4 & 18 & 19 & 21 & 14 & 16 & 15 & 17 & 22 & 13 & 23 & 24 & 20 \\ \hline
7 & 12 & 4 & 11 & 10 & 5 & 8 & 3 & 2 & 9 & 1 & 6 & 19 & 21 & 16 & 23 & 18 & 24 & 20 & 15 & 22 & 17 & 13 & 14 \\ \hline
8 & 6 & 11 & 1 & 9 & 10 & 3 & 4 & 12 & 2 & 7 & 5 & 20 & 22 & 23 & 13 & 24 & 14 & 15 & 16 & 17 & 18 & 19 & 21 \\ \hline
9 & 8 & 6 & 12 & 11 & 1 & 10 & 5 & 3 & 4 & 2 & 7 & 21 & 16 & 18 & 17 & 19 & 13 & 14 & 24 & 15 & 20 & 22 & 23 \\ \hline
10 & 7 & 12 & 2 & 4 & 11 & 5 & 6 & 8 & 3 & 9 & 1 & 22 & 23 & 24 & 18 & 20 & 19 & 21 & 14 & 16 & 15 & 17 & 13 \\ \hline
11 & 9 & 8 & 3 & 6 & 12 & 1 & 7 & 10 & 5 & 4 & 2 & 23 & 24 & 20 & 15 & 22 & 17 & 13 & 19 & 14 & 21 & 16 & 18 \\ \hline
12 & 4 & 10 & 5 & 7 & 8 & 2 & 9 & 11 & 1 & 6 & 3 & 24 & 20 & 22 & 21 & 23 & 16 & 18 & 17 & 19 & 13 & 14 & 15 \\ \hline
13 & 17 & 15 & 20 & 14 & 21 & 23 & 16 & 18 & 24 & 19 & 22 & 1 & 5 & 3 & 8 & 2 & 9 & 11 & 4 & 6 & 12 & 7 & 10 \\ \hline
14 & 13 & 17 & 22 & 15 & 16 & 24 & 18 & 19 & 20 & 21 & 23 & 2 & 1 & 5 & 6 & 3 & 8 & 9 & 10 & 11 & 4 & 12 & 7 \\ \hline
15 & 14 & 13 & 23 & 17 & 18 & 20 & 19 & 21 & 22 & 16 & 24 & 3 & 2 & 1 & 11 & 5 & 6 & 8 & 7 & 9 & 10 & 4 & 12 \\ \hline
16 & 24 & 19 & 13 & 22 & 17 & 15 & 20 & 14 & 21 & 23 & 18 & 4 & 9 & 7 & 1 & 6 & 12 & 3 & 8 & 10 & 5 & 11 & 2 \\ \hline
17 & 15 & 14 & 24 & 13 & 19 & 22 & 21 & 16 & 23 & 18 & 20 & 5 & 3 & 2 & 9 & 1 & 11 & 6 & 12 & 8 & 7 & 10 & 4 \\ \hline
18 & 20 & 21 & 14 & 23 & 13 & 17 & 22 & 15 & 16 & 24 & 19 & 6 & 4 & 9 & 10 & 7 & 1 & 12 & 2 & 3 & 8 & 5 & 11 \\ \hline
19 & 22 & 16 & 15 & 24 & 14 & 13 & 23 & 17 & 18 & 20 & 21 & 7 & 6 & 4 & 3 & 9 & 10 & 1 & 11 & 12 & 2 & 8 & 5 \\ \hline
20 & 21 & 23 & 16 & 18 & 24 & 19 & 13 & 22 & 17 & 15 & 14 & 8 & 12 & 11 & 4 & 10 & 5 & 7 & 1 & 2 & 9 & 3 & 6 \\ \hline
21 & 23 & 18 & 17 & 20 & 15 & 14 & 24 & 13 & 19 & 22 & 16 & 9 & 7 & 6 & 12 & 4 & 3 & 10 & 5 & 1 & 11 & 2 & 8 \\ \hline
22 & 16 & 24 & 18 & 19 & 20 & 21 & 14 & 23 & 13 & 17 & 15 & 10 & 8 & 12 & 2 & 11 & 4 & 5 & 6 & 7 & 1 & 9 & 3 \\ \hline
23 & 18 & 20 & 19 & 21 & 22 & 16 & 15 & 24 & 14 & 13 & 17 & 11 & 10 & 8 & 7 & 12 & 2 & 4 & 3 & 5 & 6 & 1 & 9 \\ \hline
24 & 19 & 22 & 21 & 16 & 23 & 18 & 17 & 20 & 15 & 14 & 13 & 12 & 11 & 10 & 5 & 8 & 7 & 2 & 9 & 4 & 3 & 6 & 1 \\ \hline
\end{array}\]}

We also consider the following two matrices, with entries $1$ and $-1$ that are represented, respectively, by the signs $+$ and $-$.

\[M_{\psi_1}:={\fontsize{4}{5}\selectfont
\left(\begin{array}{cccccccccccccccccccccccc}
+ & + & + & + & + & + & + & + & + & + & + & + & + & + & + & + & + & + & + & + & + & + & + & + \\
+ & + & + & - & - & + & - & - & - & - & - & + & + & - & + & + & + & + & + & - & - & + & - & - \\
+ & + & - & - & - & - & + & + & + & + & - & - & - & - & + & + & + & - & - & - & + & + & + & - \\
+ & - & - & - & - & - & + & + & + & - & + & + & - & + & + & - & + & + & + & - & - & - & - & + \\
+ & - & - & + & - & + & - & - & + & + & + & + & - & - & + & + & - & + & - & + & + & - & - & - \\
+ & - & - & - & + & + & + & - & + & + & - & + & - & + & - & + & + & + & - & + & - & - & - & - \\
+ & + & + & + & - & + & + & - & + & - & + & - & + & + & + & - & - & - & - & - & - & - & + & - \\
+ & - & + & + & + & - & - & - & + & - & + & - & - & - & - & - & + & + & + & - & + & + & + & - \\
+ & - & + & + & - & + & + & + & - & - & - & + & - & + & + & + & - & - & + & - & - & + & - & - \\
+ & + & + & - & + & - & + & - & - & + & + & - & + & - & - & - & + & - & + & + & - & - & + & - \\
+ & + & - & + & + & - & + & + & - & - & + & + & + & - & - & + & - & - & + & - & + & - & - & - \\
+ & + & - & - & - & + & - & + & - & - & + & - & + & - & + & + & - & + & + & - & - & - & + & + \\
+ & + & + & - & - & - & + & - & - & - & + & - & - & + & - & + & - & + & - & + & + & + & - & + \\
+ & + & + & - & + & + & - & + & + & + & + & + & - & - & - & - & - & - & - & - & - & + & - & + \\
+ & - & - & - & - & + & - & + & - & - & + & + & + & + & - & - & + & - & - & + & + & + & + & - \\
+ & - & - & - & - & + & + & - & + & - & - & + & + & - & - & - & - & - & + & + & + & + & + & + \\
+ & + & - & + & - & - & - & + & + & + & - & - & + & + & - & - & - & + & + & + & - & + & - & - \\
+ & - & + & + & - & - & - & + & + & - & - & + & + & - & - & + & + & - & - & + & - & - & + & + \\
+ & + & + & + & - & + & + & + & - & - & - & - & - & - & - & - & + & + & - & + & + & - & - & + \\
+ & - & + & - & + & - & - & - & + & + & - & - & + & + & + & + & - & - & + & - & + & - & - & + \\
+ & - & - & + & + & - & + & - & - & + & - & + & + & - & + & - & - & + & - & - & - & + & + & + \\
+ & + & - & + & - & + & - & - & - & + & - & + & - & + & - & - & + & - & + & - & + & - & + & + \\
+ & + & - & + & + & - & - & - & - & - & + & - & - & + & + & + & + & - & - & + & - & + & - & + \\
+ & + & + & - & + & - & - & + & - & - & - & + & - & + & + & - & - & + & - & + & + & - & + & -
\end{array}\right)}
\]
\[M_{\psi_2}:={\fontsize{4}{5}\selectfont
\left( \begin{array}{cccccccccccccccccccccccc}
+ & + & + & + & + & + & + & + & + & + & + & + & + & + & + & + & + & + & + & + & + & + & + & + \\
+ & + & + & - & - & + & + & + & + & - & - & - & + & - & + & - & + & - & + & - & - & - & + & - \\
+ & + & - & - & - & + & + & - & - & + & + & - & - & - & + & + & + & - & - & + & + & - & - & + \\
+ & + & - & - & + & - & - & + & - & + & - & + & - & + & + & - & + & - & - & - & - & + & + & + \\
+ & - & - & - & - & + & - & + & - & - & + & + & - & - & + & + & - & + & + & + & - & + & + & - \\
+ & + & + & + & + & + & - & - & - & - & + & - & - & + & - & + & + & + & - & - & - & - & + & - \\
+ & - & + & - & + & + & + & - & - & - & - & + & + & + & + & + & - & - & - & - & + & + & - & - \\
+ & - & - & + & - & + & - & - & + & + & - & + & - & + & + & - & + & + & + & - & + & - & - & - \\
+ & + & - & + & - & - & + & - & - & - & - & + & + & - & - & - & + & + & - & + & + & + & + & - \\
+ & - & + & + & - & - & - & + & - & + & - & - & + & - & + & + & - & + & - & - & + & - & + & + \\
+ & + & + & - & - & - & - & - & + & - & + & + & + & + & + & - & - & + & - & + & - & - & - & + \\
+ & - & - & + & + & - & + & + & + & - & + & - & - & + & + & - & - & - & - & + & + & - & + & - \\
+ & + & + & - & - & - & - & - & + & + & - & - & - & + & - & + & - & - & + & + & + & + & + & - \\
+ & + & + & + & + & + & - & + & - & - & - & + & - & - & - & - & - & - & + & + & + & - & - & + \\
+ & - & - & + & - & + & - & + & + & - & - & - & + & + & - & + & + & - & - & + & - & + & - & + \\
+ & - & - & - & + & + & + & - & - & + & - & - & + & + & - & - & - & + & + & + & - & - & + & + \\
+ & + & - & + & - & - & + & + & - & + & + & + & + & + & - & + & - & - & + & - & - & - & - & - \\
+ & - & + & + & + & - & - & - & - & + & + & - & + & - & + & - & + & - & + & + & - & + & - & - \\
+ & - & + & + & - & + & + & - & + & + & + & + & - & - & - & - & - & - & - & - & - & + & + & + \\
+ & - & - & - & + & - & - & - & + & - & + & + & + & - & - & + & + & - & + & - & + & - & + & + \\
+ & + & - & + & + & - & + & - & + & - & - & - & - & - & + & + & - & + & + & - & - & + & - & + \\
+ & - & + & - & + & - & + & + & + & + & - & + & - & - & - & + & + & + & - & + & - & - & - & - \\
+ & - & + & - & - & - & + & + & - & - & + & - & - & + & - & - & + & + & + & - & + & + & - & + \\
+ & + & - & - & + & + & - & + & + & + & + & - & + & - & - & - & - & + & - & - & + & + & - & -
\end{array} \right)}\]

They describe two loop-cocycles $\psi_1,\,\psi_2\in \mathcal{Z}_\mathcal{L}^2(L,\mathbb{Z}_2)$. In particular, if we define for each element $h\in L$, the cochain $\partial_h\in \mathcal{C}^1(L,\mathbb{Z}_2)$ so that $\partial_h(s)=0$ if $h\neq s$, and $\partial_h(h)=1$, then $\delta_{\psi_1}=\delta^2(f_1)$ and $\delta_{\psi_2}=\delta^2(f_2)$, where
\[f_1=\partial_2\,\partial_3\,\partial_6\,\partial_9\,\partial_{10}\,\partial_{15}\,\partial_{17}\,\partial_{19}\,\partial_{20}\,\partial_{21}\hspace{0.4cm} \text{and}\hspace{0.4cm}f_2=\partial_7\,\partial_9\,\partial_{13}\,\partial_{14}\,\partial_{15}\,\partial_{16}\,\partial_{17}\,\partial_{18}\,\partial_{19}\,\partial_{20}.\]
Both matrices $M_{\psi_1}$ and $M_{\psi_2}$ satisfy the cocyclic Hadamard test. However, $M_{\psi_1}$ is not Hadamard, while $M_{\psi_2}$ is Hadamard. More specifically, $M_{\psi_2}$ is equivalent to the Hadamard matrix \texttt{had.24.21} in the Sloane's library of Hadamard matrices \citep{Sloane}, which is known not to be cocyclic over any group \citep{OC11}. \hfill $\lhd$
\end{example}

For the first time, the loop-cocyclic Hadamard matrix $M_{\psi_2}$ in Example \ref{example_subgroup} gives an affirmative answer to Problem \ref{problem3a} outside the pseudococyclic framework. It corroborates the relevance that Moufang loops have to expand the classical cocyclic development of Hadamard matrices. In order to delve into this milestone, we have made use of Algorithm \ref{alg:HM} to compute all the Hadamard equivalence classes of loop-cocyclic Hadamard matrices over Moufang loops of order $24$. We have focused on this order because it is the smallest one in which there are Hadamard equivalence classes that are not cocyclic over groups. (Only 16 over 60 classes are cocyclic over groups \citealp{OC11}.)

\begin{algorithm}[H]
\caption{Enumeration of loop-cocyclic Hadamard matrices}\label{alg:HM}
\begin{algorithmic}[1]
\Require A loop $L$.
\Ensure Distribution of $\mathcal{Z}_\mathcal{L}^2(L,\mathbb{Z}_2)$ into Hadamard equivalence classes.

\State Compute the associative obstruction set $\mathcal{O}(L)$.
\State Compute the vector space of cocycles $\mathcal{Z}^2(L, \mathbb{Z}_2)$.
\State Compute the quotient vector space $\mathcal{B}_\mathcal{L}^k(L,\mathbb{Z}_2)$.
\State Compute the automorphism group $\mathrm{Aut}(L)$.
    \ForAll{$\phi \in \mathcal{Z}^2(L, \mathbb{Z}_2)/\mathrm{Aut}(L)$}
        \State Compute the stabilizer $Stab_{\mathrm{Aut}(L)}(\phi)$.
        \ForAll{$\psi \in \mathcal{B}_\mathcal{L}^k(L,\mathbb{Z}_2)/Stab_{\mathrm{Aut}(L)}(\phi)$}
            \If{$\phi + \psi$ is Hadamard}
                \State Store $\phi + \psi$ (if not equivalent to a previously found one).
            \EndIf
        \EndFor
    \EndFor
\end{algorithmic}
\end{algorithm}

Our computational analysis has made use of the package {\sc loops} \citep{Vojtechovsky2024} in {\sc GAP}, which allows for the systematic identification and retrieval from the built-in library of the five non-associative Moufang loops of order 24. Thus, for instance, the Moufang loop described in Example \ref{example_subgroup} is labeled as \texttt{MoufangLoop(24,3)} in the package {\sc loops}. For each algebraic structure $L$, Table \ref{tab:spaces} shows the size of its associative obstruction set (column $\mathcal{O}$), the size of its automorphism group (column $\mathrm{Aut}$), the dimension of the coboundary vector space (column $\mathcal{B}^2$), the dimension of the cocycle vector space (column $\mathcal{Z}^2$), the dimension of the quotient group $\mathcal{B}^2_\mathcal{L}(L,\mathbb{Z}_2)$ (column $\mathcal{B}^2_\mathcal{L}$), the size of the quotient group $\mathcal{Z}^2(L,\mathbb{Z}_2)/\mathrm{Aut}(L)$ (column $\mathcal{Z}^2/\mathrm{Aut}$), and the Hadamard equivalence classes (column $\mathrm{Had}$), which are again labeled according to the Sloane's library of Hadamard matrices \citep{Sloane}.

\begin{table}[ht]
\centering
\caption{Structural analysis of loop-cocycle vector spaces for all non-associative Moufang loops of order 24.}
\label{tab:spaces}
\begin{tabular}{@{}lcccccrl}
\toprule
\textbf{ID GAP} &  \textbf{$\mathcal{O}$} & \textbf{$\mathrm{Aut}$} & \textbf{$\mathcal{B}^2$} & \textbf{$\mathcal{Z}^2$}  & \textbf{$\mathcal{B}^2_\mathcal{L}$} &  \textbf{$\mathcal{Z}^2/\mathrm{Aut}$} & \textbf{$\mathrm{Had}$} (\textbf{ID \citep{Sloane}}) \\ \midrule
\texttt{MoufangLoop(24,1)} & 22 & 432 & 21 & 13 & 14 &608 & \texttt{had.24.1}$^*$\\
 &  &  &   &  &  & & \texttt{had.24.25}\\
 &  &  &   &  &  & & \texttt{had.24.27}$^*$\\
\texttt{MoufangLoop(24,2)} & 23  & 288 & 23 & 7 & 17 &48 & -\\
\texttt{MoufangLoop(24,3)} & 22 & 144 &  22 & 11 & 14 &512 & \texttt{had.24.21}\\
 &  &  &   &  &  & & \texttt{had.24.38}\\
\texttt{MoufangLoop(24,4)} & 22 & 144 &  22 & 11 & 14 &544 & -\\
\texttt{MoufangLoop(24,5)} & 22 & 432 & 22 & 10 & 14 &192 & -\\ \bottomrule
\multicolumn{8}{l}{\footnotesize $^*$ Also cocyclic over groups.}
\end{tabular}
\end{table}

Only two of the five Moufang loops of order 24 give rise to loop-cocyclic Hadamard matrices. More precisely, five Hadamard equivalence classes appear, of which two are indeed cocyclic over groups. Example \ref{example_subgroup} shows one of the remaining three classes. Similarly, the other two classes are described by the following loop-cocyclic Hadamard matrices over the corresponding Moufang loops described in Table \ref{tab:spaces}.

\[\texttt{had.24.25}\equiv {\fontsize{4}{5}\selectfont \left( \begin{array}{cccccccccccccccccccccccc}
 +& +& +& +& +& +& +& +& +& +& +& +& +& +& +& +& +& +& +& +& +& +& +& +\\
 +& -& +& -& -& +& +& +& -& -& -& +& +& +& -& +& -& -& +& -& +& -& +& -\\
 +& -& -& -& +& -& +& -& +& -& +& +& -& +& +& -& -& +& +& -& -& -& +& +\\
 +& +& -& +& +& -& +& +& -& -& -& -& +& +& +& -& +& -& -& -& +& -& -& +\\
 +& +& -& -& -& +& -& +& +& -& +& -& +& -& +& -& -& +& +& -& +& +& -& -\\
 +& +& +& -& -& -& +& -& -& +& +& -& -& -& +& -& +& -& +& +& +& -& +& -\\
 +& +& +& +& -& -& -& -& +& -& -& +& +& -& -& -& +& -& +& -& -& +& +& +\\
 +& -& -& +& -& -& -& +& -& +& +& +& -& -& -& +& +& +& +& -& +& -& -& +\\
 +& +& -& -& +& +& -& -& -& +& -& +& -& +& -& -& +& +& -& -& +& +& +& -\\
 +& -& +& +& +& +& -& -& -& -& +& -& +& -& +& +& +& +& -& -& -& -& +& -\\
 +& -& +& -& +& -& -& +& +& +& -& -& -& -& +& +& -& -& -& -& +& +& +& +\\
 +& -& -& +& -& +& +& -& +& +& -& -& -& +& +& +& +& -& +& -& -& +& -& -\\
 +& -& +& -& +& +& +& +& -& -& -& -& -& -& -& -& +& +& +& +& -& +& -& +\\
 +& +& +& -& -& -& -& +& +& -& -& +& -& +& +& +& +& +& -& +& -& -& -& -\\
 +& -& -& -& -& +& -& +& +& +& +& -& +& +& -& -& +& -& -& +& -& -& +& +\\
 +& -& +& +& -& +& -& -& -& -& +& +& -& +& +& -& -& -& -& +& +& +& -& +\\
 +& +& -& +& +& -& -& +& -& -& +& -& -& +& -& +& -& -& +& +& -& +& +& -\\
 +& -& +& +& +& -& -& -& +& +& -& -& +& +& -& -& -& +& +& +& +& -& -& -\\
 +& +& -& +& -& +& +& -& +& -& -& -& -& -& -& +& -& +& -& +& +& -& +& +\\
 +& +& +& -& -& -& +& -& -& +& +& -& +& +& -& +& -& +& -& -& -& +& -& +\\
 +& -& -& +& -& -& +& +& -& +& -& +& +& -& +& -& -& +& -& +& -& +& +& -\\
 +& -& -& -& +& -& +& -& +& -& +& +& +& -& -& +& +& -& -& +& +& +& -& -\\
 +& +& -& -& +& +& -& -& -& +& -& +& +& -& +& +& -& -& +& +& -& -& -& +\\
 +& +& +& +& +& +& +& +& +& +& +& +& -& -& -& -& -& -& -& -& -& -& -& -
  \end{array}
  \right)}\]
  \[\texttt{had.24.38}\equiv {\fontsize{4}{5}\selectfont \left( \begin{array}{cccccccccccccccccccccccc}
  +&+&+&+&+&+&+&+&+&+&+&+&+&+&+&+&+&+&+&+&+&+&+&+\\
 +&+&+&-&-&+&+&+&+&-&-&-&+&-&+&-&+&-&-&-&-&+&+&-\\
 +&+&-&-&-&+&+&-&-&+&+&-&-&-&+&-&+&-&+&+&+&-&-&+\\
 +&+&-&-&+&-&-&+&-&+&-&+&-&+&-&-&+&-&+&-&-&+&+&+\\
 +&-&-&-&-&+&-&+&-&-&+&+&-&-&+&+&-&+&-&+&-&+&+&+\\
 +&+&+&+&+&+&-&-&-&-&+&-&-&-&-&+&+&+&+&-&-&-&+&-\\
 +&-&+&-&+&+&+&-&-&-&-&+&-&+&+&+&-&-&+&-&+&+&-&-\\
 +&-&-&+&-&+&-&-&+&+&-&+&-&+&+&-&+&+&-&-&+&-&+&-\\
 +&+&-&+&-&-&+&-&-&-&-&+&+&-&-&-&-&+&+&+&+&+&+&-\\
 +&-&+&+&-&-&-&+&-&+&-&-&+&-&+&+&-&-&+&-&+&-&+&+\\
 +&+&+&-&-&-&-&-&+&-&+&+&+&+&+&-&-&+&+&-&-&-&-&+\\
 +&-&-&+&+&-&+&+&+&-&+&-&-&+&+&-&-&-&+&+&-&-&+&-\\
 +&+&+&-&-&-&-&-&+&+&+&-&-&+&-&+&-&-&-&+&+&+&+&-\\
 +&+&+&+&+&+&-&+&+&-&-&+&-&-&-&-&-&-&-&+&+&-&-&+\\
 +&-&-&+&-&+&-&-&+&-&-&-&+&+&-&+&+&-&+&+&-&+&-&+\\
 +&-&+&-&+&+&+&-&-&+&-&-&+&+&-&-&-&+&-&+&-&-&+&+\\
 +&+&-&+&-&+&+&+&-&+&+&+&+&+&-&+&-&-&-&-&-&-&-&-\\
 +&-&+&+&+&-&-&-&-&+&+&+&+&-&+&-&+&-&-&+&-&+&-&-\\
 +&+&-&-&+&-&-&+&-&-&-&-&+&+&+&+&+&+&-&+&+&-&-&-\\
 +&-&-&-&+&-&+&-&+&-&+&+&+&-&-&+&+&-&-&-&+&-&+&+\\
 +&+&-&+&+&-&+&-&+&+&-&-&-&-&+&+&-&+&-&-&-&+&-&+\\
 +&-&+&-&-&-&+&+&+&+&-&+&-&-&-&+&+&+&+&+&-&-&-&-\\
 +&-&+&+&-&-&+&+&-&-&+&-&-&+&-&-&+&+&-&-&+&+&-&+\\
 +&-&-&-&+&+&-&+&+&+&+&-&+&-&-&-&-&+&+&-&+&+&-&-
  \end{array}
  \right)}\]

Their coboundaries coincide, respectively, with the coboundaries of
\[\partial_3\,\partial_5\,\partial_7\,\partial_9\,\partial_{11}\,\partial_{14}\,\partial_{15}\,\partial_{20}\,\partial_{21} \hspace{1cm} \text{and} \hspace{1cm} \partial_{10}\,\partial_{12}\,\partial_{13}\,\partial_{14}\,\partial_{15}\,\partial_{16}\,\partial_{17}\,\partial_{19}\,\partial_{20}.\]

Note that the five non-associative Moufang loops exhibit a structural collapse of their loop-cocycle vector spaces in comparison with groups (see Table \ref{tab:spaces_2}). This is not the case for their coboundary vector spaces, whose dimensions vary similarly in both cases. Note also that while in the associative theory every coboundary is a cocycle, Table \ref{tab:spaces} shows that $\dim(\mathcal{B}^k(L,\mathbb{Z}_2))>\dim(\mathcal{Z}^k(L,\mathbb{Z}_2))$ for every non-associative Moufang loop of order 24. In particular, the vector space of cocycles collapses to dimension 7 in \texttt{MoufangLoop(24,2)}, failing completely to cover the 23-dimensional space of coboundaries.

\begin{table}[ht]
\centering
\caption{Structural analysis of cocycle vector spaces for groups of order 24.}
\label{tab:spaces_2}
\begin{tabular}{@{}lcccr}
\toprule
\textbf{ID GAP} &  \textbf{$\mathrm{Aut}$} & \textbf{$\mathcal{B}^2$} & \textbf{$\mathcal{Z}^2$}  & \textbf{$\mathcal{Z}^2/\mathrm{Aut}$} \\ \midrule
\texttt{AllSmallGroups(24)[1]} & 24 & 23 & 24& 732,160\\
\texttt{AllSmallGroups(24)[2]} & 8 & 23 & 24 & 2,150,400\\
\texttt{AllSmallGroups(24)[3]} & 24 & 24 & 24 &704,640\\
\texttt{AllSmallGroups(24)[4]} & 48 & 22 & 24 &402,176\\
\texttt{AllSmallGroups(24)[5]} & 24 & 22 & 25 &1,507,328\\
\texttt{AllSmallGroups(24)[6]} & 48 & 22& 25& 758,528\\
\texttt{AllSmallGroups(24)[7]} & 48 & 22& 25& 768,768\\
\texttt{AllSmallGroups(24)[8]} & 24 & 22& 25& 1,507,328\\
\texttt{AllSmallGroups(24)[9]} & 16 & 22& 25& 2,245,888\\
\texttt{AllSmallGroups(24)[10]} & 16 & 22& 25& 2,219,264\\
\texttt{AllSmallGroups(24)[11]} & 48 & 22& 24& 395,904\\
\texttt{AllSmallGroups(24)[12]} & 24 & 23& 25& 1,411,328\\
\texttt{AllSmallGroups(24)[13]} & 24 & 23& 25& 1,413,376\\
\texttt{AllSmallGroups(24)[14]} & 144 & 21& 27&992,768\\
\texttt{AllSmallGroups(24)[15]} & 336 & 21& 27&438,576\\
\bottomrule
\end{tabular}
\end{table}

To understand the exceptional nature of Moufang loops, we have also calculated the dimensions of both the coboundary and cocycle vector spaces for 10,000 random loops of order $24$. We obtained $\dim(\mathcal{B}^k(L,\mathbb{Z}_2))=|\mathcal{O}(L)|=24$, and $\dim(\mathcal{Z}^k(L,\mathbb{Z}_2))=1$ for all of them. This describes a base state of entropy, so any deviation in the dimension of the associative kernel in structured loops is not just a numerical difference, but a topological signal of underlying symmetry. In order to delve into this aspect, we have also computed the loop-cocyclic Hadamard matrices over all the non-associative non-Moufang right Bol loops of order $24$ in the Moorhouse's library of Bol loops \citep{Moorhouse}. Recall here that our loop $L$ is a {\em right Bol loop} if $((ab)c)b=a((bc)b)$ for all $a,b,c\in L$. It is a {\em left Bol loop} if $(a(ba))c=a(b(ac))$ for all $a,b,c\in L$. A loop is Moufang if and only if it is both a left and a right Bol loop. Table \ref{tab:spaces_3} shows our results. We highlight the existence of a new Hadamard equivalence class that is not cocyclic over any group. It is described by the following loop-cocyclic Hadamard matrix over the right Bol loop labeled as \texttt{24.15.4.1} in \citep{Moorhouse}.

 \[\texttt{had.24.2}\equiv {\fontsize{4}{5}\selectfont \left( \begin{array}{cccccccccccccccccccccccc}
-&-&-&-&-&-&-&-&-&-&-&-&-&-&-&-&-&-&-&-&-&-&-&-\\
-&+&-&-&+&+&-&+&+&-&-&+&-&+&-&+&-&+&+&-&+&-&-&+\\
-&-&-&+&-&+&-&+&-&+&+&+&+&+&+&-&-&+&+&-&-&-&+&-\\
-&-&+&+&+&-&-&-&+&+&-&+&-&+&+&+&-&-&-&+&-&-&+&+\\
-&+&-&+&-&-&+&+&+&+&-&-&-&-&+&+&+&+&-&-&+&-&+&-\\
-&+&+&-&-&+&+&-&-&+&-&+&-&+&+&-&+&-&+&+&+&-&-&-\\
-&-&-&-&+&+&-&-&-&+&+&-&-&+&-&+&+&+&-&+&+&+&+&-\\
-&+&+&-&+&-&-&+&-&-&-&-&+&+&+&+&+&-&+&-&-&+&+&-\\
-&-&+&-&-&-&-&-&+&-&+&+&-&-&+&-&+&+&+&-&+&+&+&+\\
-&-&+&+&+&+&+&+&-&-&+&+&-&-&+&+&-&-&-&-&+&+&-&-\\
-&+&-&+&+&-&+&-&+&-&+&-&-&+&+&-&-&+&+&+&-&+&-&-\\
-&+&+&+&-&+&-&+&+&+&+&-&-&+&-&-&+&-&-&-&-&+&-&+\\
-&+&+&+&-&+&-&-&+&-&+&-&+&-&-&+&-&-&+&+&+&-&+&-\\
-&+&-&+&+&-&+&-&-&-&+&+&+&+&-&-&+&-&-&-&+&-&+&+\\
-&+&+&-&+&-&-&+&-&+&+&-&+&-&+&-&-&+&-&+&+&-&-&+\\
-&-&-&-&+&+&+&-&+&+&+&-&+&-&+&+&+&-&+&-&-&-&-&+\\
-&-&+&-&-&-&+&+&+&-&+&+&+&+&-&+&+&+&-&+&-&-&-&-\\
-&-&+&+&+&+&+&+&-&-&-&-&-&-&-&-&+&+&+&+&-&-&+&+\\
-&-&-&+&+&-&-&+&+&+&-&+&+&-&-&-&+&-&+&+&+&+&-&-\\
-&+&-&-&-&-&+&+&-&+&+&+&-&-&-&+&-&-&+&+&-&+&+&+\\
-&-&+&+&-&-&+&-&-&+&-&-&+&+&-&+&-&+&+&-&+&+&-&+\\
-&-&-&-&-&+&+&+&+&-&-&-&+&+&+&-&-&-&-&+&+&+&+&+\\
-&+&+&-&+&+&+&-&+&+&-&+&+&-&-&-&-&+&-&-&-&+&+&-\\
-&+&-&+&-&+&-&-&-&-&-&+&+&-&+&+&+&+&-&+&-&+&-&+
  \end{array}
  \right)}\]

The coboundary of the associated loop-cocycle coincides with the coboundary of
\[\partial_2\,\partial_6\,\partial_8\,\partial_{10}\,\partial_{13}\,\partial_{15}\,\partial_{16}\,\partial_{19}\,\partial_{21}.\]

\begin{table}[ht]
\centering
\caption{Structural analysis of loop-cocycle vector spaces for non-associative non-Moufang right Bol loops of order 24.}
\label{tab:spaces_3}
\begin{tabular}{@{}lcccccrl}
\toprule
\textbf{ID \citep{Moorhouse}} &  \textbf{$\mathcal{O}$} & \textbf{$\mathrm{Aut}$} & \textbf{$\mathcal{B}^2$} & \textbf{$\mathcal{Z}^2$}  & \textbf{$\mathcal{B}^2_\mathcal{L}$} &  \textbf{$\mathcal{Z}^2/\mathrm{Aut}$} & \textbf{$\mathrm{Had}$} (\textbf{ID \citep{Sloane}}) \\ \midrule
\texttt{24.1.12.0} & 24 & 16 & 22 & 12 & 12 & 2,336 & -\\
\texttt{24.3.6.1} & 24 & 16 & 22 & 13  & 12 & 2,336 & -\\
\texttt{24.3.6.2} & 24 & 8 & 22 & 13 & 12 & 4,352 & -\\
\texttt{24.5.12.1} & 24 & 16 & 22 & 13  & 12 & 2,336& -\\
\texttt{24.5.12.0} & 24 & 8 & 22 & 13 & 12 & 4,352 & -\\
\texttt{24.7.6.1} & 24 & 16 & 22 & 13 & 12 & 2,336 & -\\
\texttt{24.9.4.0} & 24 & 8 & 22 & 5 & 20 & 32 & -\\
\texttt{24.7.2.0} & 24 & 16 & 22 & 5 & 20 & 20 & -\\
\texttt{24.9.7.0} & 24 & 48 & 21 & 11 & 16 & 152 & \texttt{had.24.1}$^*$\\
\texttt{24.11.6.0} & 24 & 16 & 21 & 11 & 16 & 344 & \texttt{had.24.1}$^*$\\
\texttt{24.13.5.0} & 24 & 48 & 21 & 11 & 16 & 152 & \texttt{had.24.1}$^*$\\
\texttt{24.15.4.1} & 24 & 12 & 21 & 11 & 16 & 416 & \texttt{had.24.1}$^*$\\
 &  &  &   &  &  & & \texttt{had.24.2}\\
\texttt{24.17.3.0} & 24 & 16 & 21 & 11 & 16 & 344 & \texttt{had.24.1}$^*$\\
\texttt{24.19.2.0} & 24 & 48 & 21 & 11 & 16 & 152 & \texttt{had.24.1}$^*$\\
\texttt{24.21.1.0} & 24 & 336 & 21 & 11 & 16 & 48 & \texttt{had.24.1}$^*$\\
\texttt{24.9.1.0} & 24 & 16 & 22 & 5 & 20 & 20 & -\\
\texttt{24.11.2.0} & 24 & 8 & 22 & 5 & 20  & 32 & -\\
\texttt{24.13.1.0} & 24 & 8 & 22 & 5 & 20 & 32 & -\\
\texttt{24.17.1.0} & 24 & 8 & 22 & 5 & 20 & 32 & -\\
\texttt{24.19.2.1} & 24 & 16 & 22 & 5 & 20 & 20 & -\\
\texttt{24.21.1.1} & 24 & 16 & 22 & 5 & 20 & 20 & -\\ \bottomrule
\multicolumn{8}{l}{\footnotesize $^*$ Also cocyclic over groups.}
\end{tabular}
\end{table}

We finish our study with some computational results related to the order $28$. First, in order to complement the study of loop-cocyclic Hadamard matrices presented in Example \ref{example_GS28}, we have implemented our approach to determine all loop-cocyclic Hadamard equivalence classes over the unique Moufang loop of order $28$. In addition, we have also computed all loop-cocyclic Hadamard equivalence classes over the two non-Moufang right Bol loops of order $28$ in the Moorhouse's library of Bol loops \citep{Moorhouse}. Table \ref{tab:spaces_4} shows our results.

\begin{table}[ht]
\centering
\caption{Structural analysis of loop-cocycle vector spaces for the unique Moufang loop of order 28, and the two unique non-associative non-Moufang right Bol loops of order 28.}
\label{tab:spaces_4}
\begin{tabular}{@{}lccccccl}
\toprule
\textbf{ID GAP / \citep{Moorhouse}} &  \textbf{$\mathcal{O}$} & \textbf{$\mathrm{Aut}$} & \textbf{$\mathcal{B}^2$} & \textbf{$\mathcal{Z}^2$}  & \textbf{$\mathcal{B}^2_\mathcal{L}$} &  \textbf{$\mathcal{Z}^2/\mathrm{Aut}$} & \textbf{$\mathrm{Had}$} (\textbf{ID \citep{Sloane}}) \\ \midrule
\texttt{MoufangLoop(28,1)} & 27 & 1,764 & 26 & 7 & 21 & 24 & \texttt{had.28.2}\\
&  &  &  &  &  &  &  \texttt{had.28.6}\\
&  &  &  &  &  &  & \texttt{had.28.18}\\
&  &  &  &  &  &  & \texttt{had.28.375}\\
&  &  &  &  &  &  & \texttt{had.28.410}\\
\texttt{28.9.3.0} & 28 & 12 & 26 & 5 & 24 & 20 & \texttt{had.28.376}\\
&  &  &  &  &  &  & \texttt{had.28.481}\\
&  &  &  &  &  &  & \texttt{had.28.484}\\
\texttt{28.21.1.1} & 28 & 36 & 26 & 5 & 24 & 12 & \texttt{had.28.376}\\
&  &  &  &  &  &  & \texttt{had.28.481}\\
 &  &  &   &   &  &  & \texttt{had.28.484}\\
\bottomrule
\multicolumn{8}{l}{\footnotesize $^*$ Also cocyclic over groups.}
\end{tabular}
\end{table}

In addition to the four Hadamard equivalence classes described in Example \ref{example_GS28}, we have found a new class that is not cocyclic over any group, but it is loop-cocyclic over the unique Moufang-loop of order $28$. It is described by the following loop-cocyclic matrix.
\[\texttt{had.28.6}\equiv {\fontsize{4}{5}\selectfont \left( \begin{array}{cccccccccccccccccccccccccccc}
+&+&+&+&+&+&+&+&+&+&+&+&+&+&+&+&+&+&+&+&+&+&+&+&+&+&+&+\\
+&-&-&+&-&+&-&+&+&-&-&+&+&-&+&-&+&+&-&-&+&-&+&-&+&+&-&-\\
+&-&+&-&-&+&+&-&-&-&-&+&-&+&-&+&-&-&-&+&+&+&+&+&+&+&-&-\\
+&-&+&-&-&-&-&+&-&+&+&+&+&-&-&-&+&+&+&+&-&+&-&+&-&+&-&-\\
+&-&-&+&-&-&-&-&+&+&+&-&-&+&+&-&+&-&-&+&+&+&+&+&-&-&+&-\\
+&+&-&+&+&-&-&+&-&-&-&+&-&+&+&+&+&-&-&-&-&+&-&+&-&+&-&+\\
+&+&+&-&-&-&-&-&+&-&-&-&+&+&+&+&-&+&+&-&-&-&+&+&-&+&+&-\\
+&+&-&+&-&-&+&-&-&+&-&+&+&-&-&+&+&-&+&-&-&+&+&-&+&-&+&-\\
+&-&-&-&+&-&+&+&+&+&-&-&-&-&+&+&-&+&+&-&+&+&-&+&+&-&-&-\\
+&-&+&+&+&-&+&-&-&-&+&-&+&-&+&-&-&-&+&-&+&+&+&-&-&+&-&+\\
+&-&-&-&+&+&-&-&-&+&-&-&+&+&-&-&+&-&+&-&+&-&-&+&+&+&+&+\\
+&+&+&+&-&+&-&+&-&+&-&-&-&-&+&-&-&-&+&+&-&-&+&+&+&-&-&+\\
+&+&-&-&-&+&+&+&-&-&+&-&+&+&+&+&+&-&+&+&+&-&-&-&-&-&-&-\\
+&+&-&-&+&+&-&-&+&-&+&+&-&-&+&-&-&-&+&+&-&+&-&-&+&+&+&-\\
+&-&-&+&-&-&+&+&+&-&+&-&+&-&-&+&-&-&-&+&-&-&-&+&+&+&+&+\\
+&+&-&+&-&-&+&-&+&+&-&+&-&+&-&-&-&+&+&+&+&-&-&-&-&+&-&+\\
+&-&-&-&+&-&-&-&-&+&+&+&+&+&+&+&-&+&-&+&-&-&+&-&+&-&-&+\\
+&-&+&-&-&+&+&+&+&+&-&+&+&+&+&-&-&-&-&-&-&+&-&-&-&-&+&+\\
+&+&-&-&+&+&+&-&+&-&-&-&+&-&-&-&+&+&-&+&-&+&+&+&-&-&-&+\\
+&+&-&-&-&+&-&+&-&+&+&-&-&-&-&+&-&+&-&-&+&+&+&-&-&+&+&+\\
+&-&-&+&+&+&+&+&-&-&+&+&-&+&-&-&-&+&+&-&-&-&+&+&-&-&+&-\\
+&+&+&-&-&-&+&-&-&-&+&+&-&-&+&-&+&+&-&-&+&-&-&+&+&-&+&+\\
+&-&+&+&-&+&-&-&+&-&+&-&-&+&-&+&+&+&+&-&-&+&-&-&+&-&-&+\\
+&+&+&-&+&-&+&+&+&+&+&-&-&+&-&-&+&-&-&-&-&-&+&-&+&+&-&-\\
+&-&+&+&+&+&+&-&-&+&-&-&-&-&+&+&+&+&-&+&-&-&-&-&-&+&+&-\\
+&-&+&-&+&-&-&+&+&-&-&+&-&-&-&+&+&-&+&+&+&-&+&-&-&-&+&+\\
+&+&+&+&+&+&-&-&+&+&+&+&+&-&-&+&-&-&-&-&+&-&-&+&-&-&-&-\\
+&+&+&+&+&-&-&+&-&-&-&-&+&+&-&-&-&+&-&+&+&+&-&-&+&-&+&-
  \end{array}
  \right)}\]

The coboundary of the associated loop-cocycle coincides with the coboundary of
\[\partial_3\,\partial_4\,\partial_5\,\partial_6\,\partial_{10}\,\partial_{11}\,\partial_{12}\,\partial_{14}\,\partial_{15}\,\partial_{16}\,\partial_{18}\,\partial_{19}.\]

We have also obtained the same three Hadamard equivalence classes for both right Bol loops in Table \ref{tab:spaces_4}. As an illustrative example, we focus on the right Bol loop \texttt{28.9.3.0}, over which the mentioned classes are loop-cocyclic by means of the following matrices.

\[\texttt{had.28.376}\equiv {\fontsize{4}{5}\selectfont \left( \begin{array}{cccccccccccccccccccccccccccc}
-&-&-&-&-&-&-&-&-&-&-&-&-&-&-&-&-&-&-&-&-&-&-&-&-&-&-&-\\
-&+&+&-&-&+&-&+&+&-&+&+&+&-&-&-&-&+&+&+&-&+&-&-&-&-&+&+\\
-&-&-&+&-&-&-&+&+&+&+&-&-&+&-&-&+&-&+&+&+&+&+&-&-&+&+&-\\
-&+&-&-&-&+&-&-&-&+&+&-&-&-&+&+&-&-&+&+&+&-&-&+&+&+&+&+\\
-&+&-&+&-&-&+&+&+&-&-&+&-&-&+&-&+&+&+&+&+&-&-&+&+&-&-&-\\
-&-&+&+&+&-&+&-&-&-&+&-&-&+&-&-&-&+&-&+&+&+&-&+&+&-&+&+\\
-&-&-&-&+&+&-&+&+&-&-&+&+&+&-&+&-&-&-&+&+&-&+&+&+&-&+&-\\
-&+&+&-&+&-&-&+&+&+&-&-&-&+&+&-&-&+&-&+&-&-&+&-&+&+&-&+\\
-&-&+&+&+&+&+&+&-&+&-&-&+&-&-&+&-&+&+&+&+&-&-&-&-&+&-&-\\
-&+&-&+&+&-&-&-&-&-&+&+&+&+&+&+&+&+&-&+&-&-&-&-&-&+&+&-\\
-&+&+&+&-&+&-&-&-&+&-&-&+&+&-&-&+&+&+&-&-&-&+&+&+&-&+&-\\
-&-&+&-&-&-&+&-&+&+&+&+&+&-&+&-&-&+&-&-&+&-&+&+&-&+&+&-\\
-&-&-&+&-&+&+&-&+&-&+&+&-&+&-&+&-&+&+&-&-&-&+&-&+&+&-&+\\
-&+&-&-&+&+&+&-&+&+&+&-&-&-&-&+&+&+&-&+&-&+&+&+&-&-&-&-\\
-&-&-&+&+&-&-&+&-&+&-&+&-&-&+&+&-&+&+&-&-&+&+&+&-&-&+&+\\
-&+&-&-&-&-&+&+&-&-&-&-&+&-&-&+&+&+&-&-&+&+&+&-&+&+&+&+\\
-&+&+&+&+&+&-&+&-&-&+&+&-&-&-&-&+&-&-&-&+&-&+&+&-&+&-&+\\
-&-&+&-&-&+&-&+&-&+&+&+&-&+&+&+&+&+&-&-&+&+&-&-&+&-&-&-\\
-&+&+&+&+&-&-&-&+&-&+&-&+&-&+&+&-&-&+&-&+&+&+&-&+&-&-&-\\
-&-&+&-&+&+&+&-&-&-&-&+&-&-&+&-&+&-&+&+&-&+&+&-&+&+&+&-\\
-&+&+&-&+&-&+&-&+&+&-&+&-&+&-&+&+&-&+&-&+&-&-&-&-&-&+&+\\
-&-&-&-&+&+&-&-&+&-&-&-&+&+&+&-&+&+&+&-&+&+&-&+&-&+&-&+\\
-&-&+&-&-&-&+&+&-&-&+&-&+&+&+&+&+&-&+&+&-&-&+&+&-&-&-&+\\
-&+&+&+&-&+&+&+&+&-&-&-&-&+&+&+&-&-&-&-&-&+&-&+&-&+&+&-\\
-&+&-&-&+&-&+&+&-&+&+&+&+&+&-&-&-&-&+&-&-&+&-&+&+&+&-&-\\
-&-&+&+&-&-&-&-&+&+&-&+&+&-&-&+&+&-&-&+&-&+&-&+&+&+&-&+\\
-&-&-&+&+&+&+&+&+&+&+&-&+&-&+&-&+&-&-&-&-&-&-&-&+&-&+&+\\
-&+&-&+&-&+&+&-&-&+&-&+&+&+&+&-&-&-&-&+&+&+&+&-&-&-&-&+
  \end{array}
  \right)}\]

  \[\texttt{had.28.481}\equiv {\fontsize{4}{5}\selectfont \left( \begin{array}{cccccccccccccccccccccccccccc}
-&-&-&-&-&-&-&-&-&-&-&-&-&-&-&-&-&-&-&-&-&-&-&-&-&-&-&-\\
-&+&+&+&+&-&+&-&+&+&-&-&+&-&-&+&-&-&-&+&+&-&+&+&-&+&-&-\\
-&+&-&-&-&+&+&-&+&-&+&+&-&+&+&-&-&+&-&-&+&-&+&+&-&+&-&+\\
-&-&+&-&+&-&+&-&-&-&-&-&-&+&+&-&+&+&+&+&+&-&+&+&+&-&+&-\\
-&+&-&-&+&-&+&-&+&+&-&+&+&-&-&-&-&+&+&-&+&+&-&-&+&-&+&+\\
-&-&-&-&+&-&+&+&+&+&+&+&+&+&+&+&+&-&+&-&-&-&+&-&-&-&-&-\\
-&-&+&+&+&+&+&-&-&+&-&-&-&+&+&+&-&+&+&-&-&+&-&-&-&+&-&+\\
-&+&-&+&-&-&-&+&-&+&-&-&+&+&-&-&+&+&+&-&-&-&+&+&+&+&-&+\\
-&+&+&+&+&-&-&-&-&-&+&+&+&+&+&-&-&-&-&+&-&+&+&-&+&-&-&+\\
-&-&+&-&-&+&+&+&+&+&+&-&-&-&-&-&-&-&+&+&-&+&+&+&+&-&-&+\\
-&+&+&-&-&+&-&+&+&+&-&-&+&+&+&-&+&+&-&+&+&+&-&-&-&-&-&-\\
-&-&+&+&+&-&-&+&+&-&+&+&-&-&-&-&+&+&+&+&+&-&-&-&-&+&-&+\\
-&-&-&+&+&-&-&+&+&+&+&-&-&+&+&-&-&-&-&-&+&+&-&+&+&+&+&-\\
-&+&-&+&+&+&+&+&-&-&+&-&-&-&-&+&+&+&-&-&+&+&+&-&+&-&-&-\\
-&+&+&+&-&+&+&-&+&-&+&-&+&-&+&-&+&-&+&-&-&-&-&-&+&+&+&-\\
-&-&-&+&-&+&-&-&-&+&+&+&+&-&+&+&-&+&+&+&+&-&-&+&+&-&-&-\\
-&+&-&-&-&-&-&-&-&+&+&-&-&-&+&+&+&-&+&+&+&+&+&-&-&+&+&+\\
-&-&-&+&-&-&+&-&+&-&+&-&+&+&-&+&+&+&-&+&-&+&-&+&-&-&+&+\\
-&+&-&+&+&+&+&+&-&+&-&+&-&-&+&-&+&-&-&+&-&-&-&+&-&-&+&+\\
-&-&-&+&-&+&+&+&-&-&-&+&+&+&-&-&-&-&+&+&+&+&+&-&-&+&+&-\\
-&-&+&-&+&+&-&-&-&+&+&+&+&-&-&-&+&+&-&-&-&+&+&+&-&+&+&-\\
-&+&+&-&-&-&+&+&-&+&+&+&-&+&-&+&-&+&-&+&-&-&-&-&+&+&+&-\\
-&-&+&+&-&+&-&-&+&+&-&+&-&+&-&+&+&-&-&-&+&-&+&-&+&-&+&+\\
-&+&-&-&+&+&-&-&+&-&-&+&-&+&-&+&+&-&+&+&-&+&-&+&+&+&-&-\\
-&-&-&-&+&+&-&+&+&-&-&-&+&-&+&+&-&+&-&+&-&-&+&-&+&+&+&+\\
-&-&+&-&-&-&+&+&-&-&-&+&+&-&+&+&+&-&-&-&+&+&-&+&+&+&-&+\\
-&+&+&+&-&-&-&+&+&-&-&+&-&-&+&+&-&+&+&-&-&+&+&+&-&-&+&-\\
-&+&+&-&+&+&-&+&-&-&+&-&+&+&-&+&-&-&+&-&+&-&-&+&-&-&+&+
  \end{array}
  \right)}\]

\[\texttt{had.28.484}\equiv {\fontsize{4}{5}\selectfont \left( \begin{array}{cccccccccccccccccccccccccccc}
-&+&-&-&-&-&+&+&-&+&+&-&+&+&+&-&+&-&-&+&+&-&-&+&+&-&-&+\\
-&-&+&-&+&+&+&-&+&+&+&+&-&+&-&-&+&-&+&+&-&-&-&-&+&-&+&-\\
-&+&-&-&+&-&-&-&+&+&+&-&+&+&+&-&-&+&+&-&-&+&+&-&-&-&+&+\\
-&-&+&+&+&-&+&-&-&+&-&-&+&-&+&+&-&+&+&+&-&-&-&-&+&+&-&+\\
-&+&+&+&-&+&-&+&+&-&+&+&+&+&-&-&-&+&-&-&-&-&-&-&+&+&-&+\\
-&-&-&+&+&+&+&+&-&-&-&+&-&+&+&-&+&+&+&-&+&-&+&-&-&-&-&+\\
-&+&-&-&+&+&+&+&-&-&-&+&+&-&+&-&-&-&+&-&-&+&-&+&+&+&+&-\\
-&-&-&+&-&+&-&+&+&+&+&+&+&-&+&+&-&-&+&+&-&-&+&+&-&-&-&-\\
-&-&+&+&+&-&+&-&-&+&+&+&+&+&-&-&-&-&-&-&+&+&+&+&-&+&-&-\\
-&+&-&+&-&+&+&-&-&+&-&+&-&+&-&+&-&+&-&+&-&+&-&+&-&-&+&+\\
-&+&+&+&+&-&+&+&+&-&+&-&-&-&-&+&+&-&+&-&-&+&-&+&-&-&-&+\\
-&-&-&+&-&-&+&-&+&-&+&+&+&-&+&+&+&+&-&-&+&+&-&-&+&-&+&-\\
-&-&+&-&+&+&-&+&+&+&-&-&-&+&+&+&-&+&-&-&+&+&-&+&+&-&-&-\\
-&+&-&+&-&-&+&+&+&+&-&-&-&+&-&+&-&-&+&-&+&-&+&-&+&+&+&-\\
-&+&+&-&-&+&-&-&-&+&-&+&+&-&-&+&+&-&+&-&+&+&+&-&+&-&-&+\\
-&-&+&-&-&+&+&-&+&-&-&-&+&+&+&+&+&-&-&-&-&-&+&+&-&+&+&+\\
-&-&-&+&+&+&-&+&+&+&-&-&+&-&-&-&+&-&-&+&+&+&-&-&-&+&+&+\\
-&+&-&+&+&+&-&-&-&-&+&-&-&+&+&+&+&-&-&+&-&+&+&-&+&+&-&-\\
-&+&+&+&-&+&-&-&-&+&+&-&-&-&+&-&+&+&+&-&+&-&-&+&-&+&+&-\\
-&-&+&-&-&-&-&+&-&-&+&+&-&+&+&+&-&-&+&+&+&+&-&-&-&+&+&+\\
-&-&-&-&+&-&-&+&-&+&+&+&-&-&-&+&+&+&-&-&-&-&+&+&+&+&+&+\\
-&+&-&-&+&-&-&-&+&-&-&+&+&+&-&+&+&+&+&+&+&-&-&+&-&+&-&-\\
-&+&+&-&-&-&+&+&+&+&-&+&-&-&+&-&+&+&-&+&-&+&+&-&-&+&-&-\\
-&-&-&-&-&+&+&-&+&-&+&-&-&-&-&-&-&+&+&+&+&+&+&+&+&+&-&+\\
-&-&+&+&-&-&-&+&-&-&-&-&+&+&-&-&+&+&+&+&-&+&+&+&+&-&+&-\\
-&+&+&+&+&-&-&-&+&-&-&+&-&-&+&-&-&-&-&+&+&-&+&+&+&-&+&+
  \end{array}
  \right)}\]

Their coboundaries coincide, respectively, with the coboundaries of
\[\partial_2\,\partial_4\,\partial_5\,\partial_6\,\partial_7\,\partial_9\,\partial_{13}\,\partial_{15}\,\partial_{17}\,\partial_{18}\,\partial_{23}\,\partial_{26},\]
\[\partial_2\,\partial_3\,\partial_4\,\partial_8\,\partial_9\,\partial_{10}\,\partial_{11}\,\partial_{12}\,\partial_{13}\,\partial_{15}\,\partial_{16}\,\partial_{17}\,\partial_{18}\,\partial_{21}\,\partial_{23}\,\partial_{25}\,\partial_{26}\]
and
\[\partial_2\,\partial_3\,\partial_4\,\partial_{12}\,\partial_{13}\,\partial_{14}\,\partial_{15}\,\partial_{19}\,\partial_{23}\,\partial_{26}.\]

\section{Further work}

This paper has established a formal framework for the cocyclic development of Hadamard matrices over loops by developing a cohomology with associativity obstructions. On the basis of this theoretical foundation, we have formalized the notion of loop-cocyclic Hadamard matrices, demonstrating that pseudococyclic matrices are naturally embedded within this framework. We have also formalized the relationship among loop-cocycles and central extensions of loops, and studied how these extensions can be implemented into the cocyclic development of Hadamard matrices over loops.

Overall, our cocyclic framework opens new possibilities for both theoretical and computational developments in combinatorial design theory. In this last regard, we have proved the existence of four new Hadamard equivalence classes of order 24 and eight others of order 28 that are not cocyclic over any group, but they are loop-cocyclic over Moufang loops or right Bol loops. An immediate theoretical challenge emerges in particular from the distribution of combinatorial designs observed in Table \ref{tab:spaces_3}. The computational results therein shown reveal that only a selective subset of non-associative non-Moufang right Bol loops admit loop-cocyclic Hadamard developments. The systematic collapse of the cocycle space to $\dim(\mathcal{Z}^2) = 11$ among all candidates strongly implies the existence of an underlying rigid algebraic invariant. Future research must determine whether this dimensionality constraint constitutes a sufficient condition for admissibility, or if it reflects a deeper obstruction involving the core or the nucleus of the loop variety.

It remains an open objective to extend this computational framework to other families of non-associative loops and to consistently analyze common algebraic properties of those loops that have successfully generated loop-cocyclic Hadamard matrices. This analysis aims to establish conditions to guide the selection and algebraic construction of new candidate families of loops over which new loop-cocyclic Hadamard equivalence classes can be found.

Concerning the computational study of higher orders, it is necessary to deal with the inconvenience combinatorial explosion inherent to orbit-counting over dense coboundary spaces. In this regard, Theorem \ref{theorem_aut} should be reinterpreted through representation theory. The distribution of stabilizer sizes within the orbit summation is fundamentally non-random. It reflects the explicit structure of the cocyclic vector space viewed as a finite module over the group algebra $\mathbb{F}_2[\text{Aut}(L)]$. Decomposing this module into its primary irreducible constituents would yield a more efficient approach.

In any case, the computational evidence presented herein demonstrates that the classical cocyclic Hadamard conjecture is artificially restricted by the rigidity of the associative locus. The natural generalization of the cocyclic development of Hadamard matrices from groups to loops enables us to propose a more general framework: for every integer $t \ge 1$, there exists a loop-cocyclic Hadamard matrix developed over a loop of order $4t$. This formulation shifts the focus from classical group cohomology to a systematic classification of coboundary operators in non-associative extensions.

Finally, while the formal cohomological framework developed in Section \ref{sec:preliminaries} is valid for arbitrary dimensions $k$, our current computational implementations are constrained to the case $k=2$. We propose as further work to delve into this computational approach for higher-degree spaces.

\section*{Acknowledgement}

The first and third authors have partially been supported by the Research Project {\em HADAMARD-CRYPT} (PID2024-160735NB-I00), co-financed by the EU -- Ministry of Finance and Public Administration -- European Funds -- Ministry of University, Research and Innovation. The second author has partially been supported by the Research Project {\em ``Modeling small-world, scale-free networks from combinatorial designs based on quasigroup digraphs''} (PPIT-FEDER-SOL2024-31611), co-financed by the EU -- Ministry of Finance and Public Administration -- European Funds -- Andalusian Regional Government -- Ministry of University, Research and Innovation.

\end{document}